\documentclass[11pt]{article}
\usepackage[utf8]{inputenc}

\usepackage[final]{graphicx}

\usepackage{xurl}

\newcommand{\highlightcol}{blue}
\usepackage{hyperref}
\hypersetup{
  colorlinks   = true,    
  urlcolor     = \highlightcol,    
  linkcolor    = \highlightcol,    
  citecolor    = \highlightcol     
}

\usepackage[a4paper, textwidth={14cm}]{geometry}

\usepackage{amsthm}
\usepackage{amsmath}
\usepackage{amsfonts}
\usepackage{amssymb}
\usepackage{url}

\newtheorem{theorem}{Theorem}[subsection]
\newtheorem{proposition}[theorem]{Proposition}
\newtheorem{definition}[theorem]{Definition}
\newtheorem{example}[theorem]{Example}
\newtheorem{corollary}[theorem]{Corollary}
\newtheorem{lemma}[theorem]{Lemma}
\theoremstyle{remark}
\newtheorem{remark}[theorem]{Remark}

\newcommand{\keywords}[1]{{\bf Keywords:} #1}

\begin{document}

\title{Formalising the intentional stance 2:\\ a coinductive approach}

\author{Simon McGregor${}^1$, timorl, Nathaniel Virgo${}^{2,3}$ \\[1ex]
\small
${}^1\,$University of Sussex, UK \\
\small
${}^2\,$Centre of Data Innovation Research,\\
\small School of Physics, Engineering \& Computer Science,\\ 
\small University of Hertfordshire, UK\\
\small
${}^3\,$Earth-Life Science Institute,\\ \small Institute of Science Tokyo, Japan.
}

\maketitle

\begin{abstract}
	We are concerned with the mathematical foundations of cognition: in particular, given a stochastic process with inputs and output, how might its behaviour be related to the pursuit of a goal?
We model this using what we term \emph{transducers}, which are a mathematical object that captures only the external behaviour of such a system and not its internal state. A companion paper provides an accessible description of our results aimed at cognitive scientists, while the current paper gives formal definitions and proofs of the results.

To formalise the concept of a system that behaves as if it were pursuing a goal, we consider what happens when a transducer (a `policy') is coupled to another transducer that comes equipped with a success condition (a `teleo-environment').
A (globally) optimal policy is identified with the behaviour of a system that behaves as if it were perfectly rational in the pursuit of a goal; our framework also allows us to model constrained rationality.

We find that (globally) optimal policies have a property closely related to Bellman’s principle from dynamic programming: a policy that is optimal in one time step will again be optimal in the next time step, but with respect to a different teleo-environment (obtained from the original one by a modified version of Bayesian filtering).
This helps to elucidate the appearance of Bayesian formalisms in models of cognition.
We describe a condition that is sufficient for this property to also apply to the bounded-rational case.

Additionally, we show that a policy is deterministic if and only if there exists a teleo-environment for which it is uniquely optimal among the set of all policies; this is at least conceptually related to classical representation theorems from decision theory.
This need not hold in the bounded-rational case; we give an example of this related to the so-called absent-minded driver problem.
All of the formalism is defined using coinduction, following the style proposed by Czajka\cite{CoinductionCzajka15}.

\end{abstract}

\newcommand{\Z}{\mathbb{Z}}
\newcommand{\N}{\mathbb{N}}

\newcommand{\transducerto}{\mathrel{\triangleright}}

\newcommand{\inputSpace}{I}
\newcommand{\outputSpace}{O}
\newcommand{\indexingSpace}{T}
\newcommand{\transducer}{\pi}
\newcommand{\emptyString}{\varepsilon}
\newcommand{\interleaved}[2]{\left( #1:#2 \right)}
\newcommand{\explicitTransducers}[3]{#2 \transducerto_{#1} #3}
\newcommand{\stochasticProcesses}[2]{\mathrm{Stoch}\left(#1, #2\right)}
\newcommand{\transducers}[2]{#1 \transducerto #2}
\newcommand{\explicitGenericTransducers}{\explicitTransducers{\indexingSpace}{\inputSpace}{\outputSpace}}
\newcommand{\genericTransducers}{\transducers{\inputSpace}{\outputSpace}}
\newcommand{\distributions}[1]{P\left( #1 \right)}
\newcommand{\probability}{\mathbb{P}}
\newcommand{\transition}{\mathbb{T}}
\newcommand{\elementString}[1]{\mathbf{#1}}
\newcommand{\supp}{\mathrm{supp}}
\newcommand{\evolution}{\bullet}

\newcommand{\constrainedTransducers}{T}

\newcommand{\stateSpace}{S}
\newcommand{\actionSpace}{A}
\newcommand{\telosSpace}{G}
\newcommand{\noSuccess}{\bot}
\newcommand{\success}{\top}
\newcommand{\environment}{\varepsilon}
\newcommand{\environments}{\transducers{\actionSpace}{\left( \stateSpace \times \telosSpace \right)}}
\newcommand{\policy}{\transducer}
\newcommand{\policies}{\transducers{\stateSpace}{\actionSpace}}
\newcommand{\successProbability}{S}

\newcommand{\nothingDistribution}{U_{\noSuccess}}
\newcommand{\successDistribution}{U_{\success}}
\newcommand{\doomEnvironment}{\environment_0}

\makeatletter
\DeclareRobustCommand{\btleft}{\mathbin{\mathpalette\btlr@\blacktriangleleft}}
\DeclareRobustCommand{\btright}{\mathbin{\mathpalette\btlr@\blacktriangleright}}

\newcommand{\btlr@}[2]{%
  \begingroup
  \sbox\z@{$\m@th#1\transducerto$}%
  \sbox\tw@{\resizebox{1.1\wd\z@}{1.1\ht\z@}{\raisebox{\depth}{$\m@th#1\mkern-1mu#2$}}}%
  \ht\tw@=\ht\z@ \dp\tw@=\dp\z@ \wd\tw@=\wd\z@
  \copy\tw@
  \endgroup
}

\keywords{mathematical formulations of agency, controlled stochastic processes, as-if agency, intentional stance, POMDPs, applications of coinduction}

\section{Introduction}

This paper presents a formal mathematical treatment related to
  Dennet's \emph{intentional stance} \cite{Dennett1975, Dennett1981,
    Dennett2006}. A companion paper provides motivation for this work, and discusses the approach at a conceptual level,
  while the current paper is focused on the formal aspects and detailed mathematical results.

  To summarise the motivation discussed in the companion paper, we start with a
  system that has inputs and outputs and behaves stochastically. We
    formalise this notion in Section \ref{sec.transducers} under the name of a \emph{transducer}.
  Transducers model the externally observable behaviour of a system, without saying anything about any internal state it might have.

  Transducers are defined coinductively: a transducer is {defined as} a probability distribution $p$ over an output alphabet $O$, together with a function that takes an element of $I\times \supp(p)$ and returns a new transducer, where $I$ is its input alphabet.
  The interpretation is that a transducer {is a gadget that} first gives an output stochastically and then takes an input, returning a new transducer, which can in general depend both on the input it received and on the output it gave.
  This process of generating a new transducer can be thought of as conditioning the input-output behaviour of a system on its input and output, resulting in a new ``externally observable behaviour'' that takes place over the remaining time steps.
  We call this process of updating \emph{evolution} by the input-output pair $(i,o)$.

  We start with a transducer representing the the externally observable behaviour of a system that we want to regard as an agent in some way.
  We are interested in what it would mean to treat such a system ``as if'' it were an agent, following an approach broadly along the lines of previous work on ``as if'' agency \cite{McGregor2016,McGregor2017,VirgoBiehlMcGregor21}, as described in the companion paper.
  To do this we introduce a notion of \emph{teleo-environment}, which is
  another stochastic system that can be coupled to the original one.
  It consists not only of an environment that the system interacts with but also a specification of a goal.
  In our case the goal comes in the form of a ``sucess'' event that may or may
  not occur on each time step, with a probability that depends on the history of
  the two systems' interactions. This notion of optimality is
    defined and explored in Section \ref{sec.teleo}.

  Teleo-environments could be seen as a kind of partial observable Markov decision process (POMDP), while the systems they couple to could be seen as candidate solutions.
  However, our main interest is not in finding a policy that is optimal for a
  given teleo-environment, but rather in exploring the optimality
  relation between policies and environments in general.

Specifically, we are interested in how the relationship between optimal policies and environments for which they are optimal evolves over time.
It turns out that, in the absence of additional constraints (which we will address shortly), the evolution of an optimal policy $\pi$ for an environment $\varepsilon$ is always optimal for a corresponding evolution of that environment.
This can be seen as a version of \emph{Bellman's principle} from dynamic programming and reinforcement learning.
Perhaps surprisingly, it can also be seen as a kind of \emph{Bayesian filtering}.
Bellman's principle states that in certain settings, if a candidate solution is optimal at one time step it will still be optimal in the next time step (conditioned on its input and output).
Bayesian filtering concerns the updating of priors to posteriors, also conditioned on inputs.
In Bayesian filtering, one keeps a prior over the current value of some hidden variable.
At each time step, one performs Bayesian updating to obtain a posterior over the \emph{new} value of the hidden variable, so that previous values are forgotten.
The relationship between Bellman's principle and Bayesian filtering in our
context is explained in Section \ref{sec.filtering}.

However, in our framework the version of Bellman's principle that applies is slightly different to the familiar one from dynamic programming.
In addition to conditioning on the system's input and output, the applicable version of Bellman's principle also involves conditioning on the event that success does not occur (even if success may in fact have occurred).
We call this the \emph{value-laden Bellman principle for teleo-optimality} (Definition \ref{def.valueLadenBellman}), and it gives rise to a corresponding notion of \emph{value-laden (Bayesian) filtering}.
The reason for this has to do with the way we have defined optimality, whereby an agent is optimal if it maximises the probability that success occurs {at least once}.
The details are explored in Section \ref{sec.filtering}.
We also show that if a policy is (unconstrained) optimal for some environment, there must be at least one environment for which a non-value-laden version of the Bellman property holds, corresponding to a more intuitive `sensorimotor-only' version of Bayesian filtering. This is not the case in general for constrained optimality.
Real agents do not have unbounded computational resources, and consequently are not necessarily globally optimal for the problems they try to solve. Decision-making with limited resources is often known as \emph{bounded rationality}, and it forms an important part of our framework. There are two main formal approaches to bounded rationality. In the terminology of \cite{Icard-draft}, the `cost-theoretic approach' incorporates a resource cost term into the cost function, while the `panoramic' approach considers optimality within some constrained class of solutions. We follow the panoramic approach here.

In our case, a ``constrained class'' is nothing but a set of transducers; a transducer is optimal within a constrained class (for a given teleo-environment) if it performs at least as well as every other member of the class.
We show, perhaps not surprisingly, that our version of Bellman's principle does not apply to optimality within an arbitrary constrained class of transducers.
However, in Section \ref{sec.filtering} we show that if a constraint class has a property we call `closed under splicing', the value-laden Bellman property still holds.
Not all constrained classes of interest have this property, however, and in Example \ref{ex.noUfsFiltering} we give a counterexample that does not obey the value-laden Bellman property; the counterexample involves memory constraints.

We are also interested in when a transducer is \emph{uniquely} optimal for some teleo-environment.
In such a case we can say the teleo-environment \emph{specifies} the system's behaviour, i.e.\ if we know the system's `goal' then we can say what the system will do.
If a transducer is uniquely optimal for some teleo-environment, we say it is \emph{specifiable}.
We show that the deterministic transducers are exactly the transducers that are uniquely optimal (in the global, unconstrained sense) for some teleo-environment.
(Though there may be, and in general are, many teleo-environments for which a given deterministic transducer is uniquely optimal.)
In Section \ref{ssec.specifyingEnv} we also give some necessary conditions on a teleo-environment that must hold in order for it to have a uniquely optimal solution.
We also investigate specifiability in the case of constrained optimality.
If a constrained class includes all the deterministic transducers then they are still exactly the specifiable transducers.
However, in a constrained class that lacks all deterministic transducers, there may be non-deterministic transducers that are specifiable within that class (i.e. uniquely optimal within the constrained class, for some teleo-environment).
We give an example of this in Section \ref{sec.absentMinded}, based on the \emph{absent-minded driver} from game theory \cite{PICCIONE19973}, which is uniquely optimal within the class of memoryless strategies.

We start with a short commentary on \emph{coinduction}, the formal methodology we use for our analysis.

\subsection{The coinductive approach}
\label{sec.coinductionForDummies}

Our treatment uses \textit{coinductive} reasoning to construct the objects we call transducers. While the coinductive formalism might not be familiar to most readers, it corresponds closely to the intuition of `moving step-by-step' and produces elegant proofs. For any readers wishing to {become familiar with the formal details of} coinductive reasoning we recommend \cite{CoinductionCzajka15} for a set-theoretic justification for the approach we use in this paper, which treats coinduction as a first-class primitive. Alternatively, \cite{IntroCoalgebraJacobs16} provides a
comprehensive introduction from the perspective of category theory, and \cite{HoTTCoinductionAhrensEtAl15} a type-theoretic approach.

Intuitively, while induction lets us define and reason about objects with a finite internal structure, such as natural numbers or finite trees, coinduction lets us define and reason in a similar way about objects with an infinite internal structure. In practice this amounts to accepting objects that do not start at an inductive base case, such as trees without leaf nodes, which are necessarily infinite. Coinduction is the category-theoretic dual of mathematical induction, and therefore many of the concepts associated with  induction have coinductive analogs. The main mathematical consequence of this is that one can define functions \emph{into} coinductive objects using the function being defined \cite[Theorem~2.4]{CoinductionCzajka15}; \cite[Theorem~7]{HoTTCoinductionAhrensEtAl15} and prove theorems about relations that are coinductively defined \cite[Theorem~4.29]{CoinductionCzajka15}\footnote{No theorem for this is needed in \cite{HoTTCoinductionAhrensEtAl15} due to type theory treating propositions as types.} or equality \cite[Theorem~4.33]{CoinductionCzajka15} \cite[Theorem~18]{HoTTCoinductionAhrensEtAl15} using coinductive hypotheses (by analogy to inductive hypotheses).

Note that we could have derived the same results without using coinduction, but we chose to use the coinductive formalism in order to emphasise the dynamical relations between transducers, permitting an  elegant set of definitions and proofs.

\section{Relation to previous work}

The relationship of our work to cognitive science is discussed in detail in the companion paper.
Here we restrict ourselves to a brief overview of previous approaches to `as-if' agency, followed by previous works that have considered transducer-like concepts, and finally the relationship between our work and computational mechanics.

\subsection{``As-if'' agency}

The current work is part of a research programme of ``as-if'' agency \cite{McGregor2016,VirgoBiehlMcGregor21}, which takes inspiration from Dennett's \cite{Dennett1975,Dennett1981Book,Dennett2006} idea of the \emph{intentional stance} in that it aims to understand in mathematical language what it means to treat a system as if it is an agent.
In \cite{VirgoBiehlMcGregor21} this is approached via the notion of an \emph{interpretation map}, which maps the system's internal state to a Bayesian prior over the states of another system that it can partially observe, subject to equations that guarantee the prior will be updated to a posterior in a consistent way.
In \cite{Biehl2022} this is extended to the case of POMDPs, where the agent must interact with its environment to maximise an exponentially discounted expected reward.
However, the current paper takes a somewhat different approach, since we are concerned with interpreting only the externally observable behaviour of a system, without any reference to its internal state.

The current work shares with these previous works the idea that the intentional stance is optional, and that there may be (and in general are) many different ways to interpret a given system as an agent.

This can be contrasted with the approach taken in \cite{Orseau2018}, which is related in that it starts from a given behaviour and asks whether it is an agent in the sense of being optimal for \emph{some} POMDP-like task, but differs in that its focus is on empirically distinguishing agents from non-agents, assuming that one already has well defined classes of each. With this assumption they are able to describe some specific algorithms for distinguishing agents from non-agents and inferring their goals. Our current work, among other things, calls into question the idea that agents and non-agents can in principle be distinguished in this way.

\subsection{Transducers and input-output behaviours}

The central object of our paper, the `transducer', is meant to model the externally observable input-output behaviour of a system, without reference to its internal state.
We want to emphasise that this idea by itself is not new.
We show in Section~\ref{sec.unrolling} that transducers could equivalently be defined as probability distributions of infinite sequences of outputs given infinite sequences of inputs, subject to a causality condition that prevents past outputs from depending on future inputs.
This idea appears throughout \cite{Hutter2005universal} under the name of chronological distributions.
The context here is quite close to ours in that it deals with POMDP-like tasks and their optimal solutions.
Such conditional distributions over sequences also appear in information, such as in the definition of directed information \cite{massey1990causality}, where they are called discrete channels with memory.
Similar ideas also appear in control theory under the name of ``controlled stochastic process,'' as mentioned in \cite{DiLavore2022coinductive}.

Our reason for introducing a new name is to emphasise the notion of evolution, along with our coinductive way of thinking about them.
That is, we use the term `transducer' to emphasise that conditioning on one or more time steps of given inputs and outputs results in a new distribution over future outputs given future inputs, namely the evolved transducer.

We do not claim originality over the coinductive approach either however, since a coinductive approach to controlled stochastic processes is also taken by \cite{DiLavore2022coinductive}, in a rather more general formal setting than ours, making heavy use of category theory.
The originality of our work lies in its results and what they say about Dennett's ideas, rather than in the details of our formalism.

\subsection{Computational mechanics and unifilar machines}

As in computational mechanics we are concerned with the externally observable behaviour of stochastic input-output processes, and the way in which we do this has something in common with the idea of an $\varepsilon$-transducer \cite{barnett2015-epsilon}.
The formal objects known as $\varepsilon$-machines and $\varepsilon$-transducers are the minimal unifilar representations of stochastic processes (usually assumed to be stationary), and their states are called causal states.
The concept of unifilarity plays an important role in our work (see Section~\ref{sec.unifilarMachines}).
This concept originally comes from the computational mechanics literature, since it is used in one of the ways that one can define an $\varepsilon$-machine \cite{travers2011equivalence}, or more generally an $\varepsilon$-transducer \cite{barnett2015-epsilon}.
However, in our case, instead of using a minimal unifilar representation we are  concerned with the equivalent of causal states as mathematical objects in their own right (our transducers).
The connection between these ideas is made more precise in \cite{Virgo2023unifilar}, where the set of all transducers is shown to form the state space of a single unifilar machine.
(The terminal one, in the sense of category theory.)
The mapping from stochastic Moore machines to unifilar machines (our definition~\ref{def.unifilarMapping}) has much in common with the `mixed state representation' of a stochastic processes described in \cite{crutchfield2016complexity}.
Although our work is related to computational mechanics in that unifilarity is a central concept, it differs in that we are generally not concerned with stationary processes (indeed, it's not entirely clear what form a stationarity assumption would take in our framework) or with minimal representations, so the questions we address are somewhat different from the typical computational mechanics literature.

\section{Transducers}
\label{sec.transducers}

To mathematically describe the environments and policies we will be discussing in this paper we will use the basic notion of a \emph{transducer}.
Conceptually, as mentioned in the introduction, a transducer is something that takes a sequence of inputs and returns a sequence of outputs, which may depend stochastically on the inputs.
This dependence must be \emph{causal}, in the sense that a given output can only depend on inputs that were received at earlier times.

However, in order to reason formally about transducers we will instead think of them in a different way.
We think of a transducer as something that can return an output stochastically.
Having received its output we can then give it an input, at which point the transducer will transform into a new transducer, which will in general have a different probability distribution for its next output.
In this way we can give the transducer a sequence of inputs and receive a sequence of outputs, and the causality condition will be satisfied automatically.

We will employ this notion of transducer in several different ways.
We will use transducers to model the input-output behaviour of the systems we are interested in, but we will also use them to model those agents' supposed {beliefs} about their environment. It will turn out that we can also use transducers to represent stochastic mixtures of other transducers - we will explain this concept below.

Our concept of transducer has something in common with the notion of {$\varepsilon$-transducer} defined in \cite{barnett2015-epsilon}, in that both model a stochastic process that takes a sequence of inputs and generates a sequence of outputs, with each output being able to depend on inputs that precede it.
However, the $\varepsilon$-transducer framework is concerned with stationary processes, which leads to a substantially different formalism.
Our transducers also have much in common with controlled stochastic processes and other similar concepts from artificial intelligence and control theory.
For example, the \emph{chronological semimeasures} in \cite{Hutter2005universal} are an equivalent mathematical object and are also used to model the policies and environments of agents.
Our main reason for coining a new term is our focus on the coinductive description, and in particular on the notion of evolving a transducer described in Section~\ref{sec.evolutions}.

We will define transducers coinductively, as mentioned in Section~\ref{sec.coinductionForDummies}.
This allows us to use coinductive proofs, presented in the style proposed in \cite{CoinductionCzajka15}. We will also introduce some additional notation particularly useful for this style.

Now on to define our main object of interest.
In the following, given a finite set $S$, we write $P(S)$ for the set of all probability distributions over $S$, that is, the set of functions $p\colon S\to [0,1]$ such that $\sum_{a\in A}p(a) = 1$.
Given a probability distribution $p\in P(S)$, we write $\supp(p)$ for the support of $p$, which is a subset of $S$.

\begin{definition}[Transducer]
  \label{def.transducer}
  Let $\outputSpace$ and $\inputSpace$ be finite sets called the \emph{output space} and \emph{input space} respectively. We define \emph{the set of transducers from $\inputSpace$ to $\outputSpace$, written $\genericTransducers$,} coinductively using the single constructor taking $p \in \distributions{\outputSpace}$ and $t \colon \inputSpace \times \supp\left( p \right) \to \genericTransducers$ written as $\left( p, t \right)$.
\end{definition}

Given a transducer $\transducer \colon \genericTransducers$ we will write $\probability_{\transducer}$ for the probability distribution in the first argument of the constructor, and $\transition_{\transducer}$ for the transition function in the second argument. The functions $\probability$ and $\transition$ are sometimes called \emph{destructors}.

Briefly, this means that a transducer $\pi$ is an object that provides two things:
first, it provides a way to stochastically generate a member of the output set, given by the distribution $\probability_\pi$.
Second, it provides a way to obtain a new transducer, given an input $i\in I$ and an output $o\in \supp(\probability_\pi)$.
We can think of a transducer as specifying a conditional distribution over output sequences given input sequences.
If we have given the transducer an input $i$ and we have observed that it produced the output $o$ in response, then we can calculate a new conditional distribution over output sequences given input sequences, starting from the next time step.
The result is again a transducer, and in fact it is the one given by $\transition_\pi(i,o)$.
The reason the transition function takes an element of $\supp(\probability_\pi)$ rather than $O$ is that we can only calculate conditional distributions for those outputs that can occur with positive probability.
We explore this connection more in section~\ref{sec.unrolling}, where we also show that this operation can be seen as a form of Bayesian conditioning.

For readers familiar with the category-theoretic approach to coalgebra  as in \cite{IntroCoalgebraJacobs16}, a transducer is an element of the final coalgebra of the polynomial functor
$\sum_{p\in P(O)}(I\times \supp(p)\to {-})$, or $\sum_{p\in P(O)}y^{I\times \supp(p)}$ in the notation of \cite{niu2023poly}.
An explicit construction of this terminal coalgebra is given in \cite{Virgo2023unifilar}.
However, we will generally not mention the category-theoretic approach,
preferring instead to reason in the style proposed in \cite{CoinductionCzajka15}, resembling classical inductive reasoning.

If we have a countable set of probability distributions, together with a probability weight assigned to each one, we can form a mixture distribution, which is itself a probability distribution.
In the same way, if we have a sequence of transducers together with a sequence of probability weights assigned to them, we can form a mixture of the transducers, which is itself a transducer.
One application for this is the case where we are dealing with an unknown transducer, drawn from a distribution with countable support.

\newcommand{\explicitMixture}[3]{\sum_{#1 = 1}^{\infty} #2{#1} #3{#1}}
\newcommand{\mixture}[2]{\explicitMixture{k}{#1}{#2}}
\newcommand{\kthProbability}[2]{\probability_{#1_{#2}}}
\newcommand{\kthTransition}[2]{\transition_{#1_{#2}}}
\newcommand{\kthProbabilityAt}[3]{\kthProbability{#1}{#3}\left( #2 \right)}
\newcommand{\kthTransitionAt}[4]{\kthTransition{#1}{#4}\left( #2, #3 \right)}
\newcommand{\kthCoefficientAt}[4]{\frac{#1_{#4}\kthProbabilityAt{#2}{#3}{#4}}{\explicitMixture{l}{#1_}{\kthProbabilityAt{#2}{#3}}}}
\begin{definition}[Mixture of transducers]
  \label{def.transMixture}
  Given transducers $\transducer_{k} \colon \genericTransducers$ and numbers $\alpha_{k} \geq 0$, such that $\sum_{k =
    1}^{\infty} \alpha_{k} = 1$, we define a (weighted) mixture $\sigma = \mixture{\alpha_}{\transducer_}$ as a transducer $(\probability_\sigma, \transition_\sigma)$ with
  \begin{equation}
    \probability_{\sigma} = \mixture{\alpha_}{\kthProbability{\transducer}}
  \end{equation}
  and
  \begin{equation}
    \label{eq.mixtureUpdate}
    \transition_{\sigma}(i,o) = \mixture{\kthCoefficientAt{\alpha}{\transducer}{o}}{\kthTransitionAt{\transducer}{i}{o}},
  \end{equation}
  for $i\in\inputSpace, o\in \supp(\probability_{\sigma})$.
\end{definition}
Note that this is a coinductive definition, since Equation~\ref{eq.mixtureUpdate} uses the notion of a mixture of transducers on its right-hand side.

When only a finite number of the coefficients are nonzero we will write these mixtures as ordinary sums, omitting the zero terms.

It is not a coincidence that Equation~\ref{eq.mixtureUpdate} resembles a Bayesian update.
A mixture of transducers can be interpreted in Bayesian terms as a prior over these transducers (or technically, over the indices of a list of transducers) with the coefficients being the associated probabilities.
If we feed the unknown transducer an input $i$ it will emit an output $o$ and evolve into a new transducer.
When we receive the output $o$ we must use Bayes' theorem to update our prior to a posterior, which is what Equation~\ref{eq.mixtureUpdate} expresses.

\subsection{Evolutions}
\label{sec.evolutions}

Intuitively a transducer $\transducer \colon \genericTransducers$ outputs an element $o \in \outputSpace$ with probability $\probability_{\transducer}\left( o \right)$ and receives an input element $i \in \inputSpace$, turning into the transducer $\transition_{\transducer}\left( i, o \right)$ in the process. This operation is the basis of all the behaviours we investigate in this paper, and one of the main reasons for choosing the definitions we are working with is how naturally it arises within this context.

In this section we extend this operation from one time step to many and introduce notation for it, as well as some related necessary notation. We start with the latter.
\begin{definition}[Notation for strings]
  \label{def.stringNotation}
  Given a set $A$ we will consider the set of finite strings of elements of this set $A^{*}$. We will write $a, b,\dots$ for elements of $A$ and $\elementString{a}, \elementString{b},\dots$ for elements of $A^{*}$.

  We will denote concatenation of two strings by $\elementString{a} \cdot \elementString{b}$, overloading the operation for single elements as well, $\elementString{a} \cdot b$. We also denote the empty string as $\emptyString$.

  We denote the element of a string at position $n$ by $\elementString{a}_{n}$ (with the indexing starting from 0), and prefixes and suffixes as $\elementString{a}_{<n}$ and $\elementString{a}_{>n}$ respectively.

  Given two strings of the same length $\elementString{a} \in A^{n}, \elementString{b} \in B^{n}$ we write $\interleaved{\elementString{a}}{\elementString{b}} \in \left( A \times B \right)^{n}$ for the interleaved string $\left( \left(\elementString{a}_{i}, \elementString{b}_{i}\right) \right)_{i = 0}^{n}$. We will omit any additional parentheses when working with more than two strings, just as we do with tuples.
\end{definition}

With this background we can define transducer evolutions.
\newcommand{\undefinedEvolution}{\star}
\begin{definition}[Evolution]
  \label{def.evolution}
  We define the operation of transducer evolution by finite trajectories using the infix operator $\evolution \colon \genericTransducers \to \left(\inputSpace \times \outputSpace\right)^{*} \to \genericTransducers \cup \left\{ \undefinedEvolution \right\}$ recursively with the base
  \begin{equation}
    \transducer \evolution \interleaved{\emptyString}{\emptyString} = \transducer,
  \end{equation}
  and the step
  \begin{equation}
    \transducer \evolution \interleaved{ \elementString{i} \cdot i}{ \elementString{o} \cdot o } =
    \begin{cases}
      \transition_{\transducer \evolution \interleaved{ \elementString{i}}{ \elementString{o} }}\left( i, o \right) & \textrm{ if } \transducer \evolution \interleaved{ \elementString{i}}{ \elementString{o} } \neq \undefinedEvolution \textrm{ and } o \in \supp\left( \probability_{\transducer \evolution \left( \elementString{i}, \elementString{o} \right)} \right),\\
      \undefinedEvolution & \textrm{ otherwise.}
    \end{cases}
  \end{equation}
  We will refer to an evolution returning a transducer (as opposed to $\undefinedEvolution$) as `valid'.
\end{definition}

\subsection{Unrolling transducers}
\label{sec.unrolling}

We present an alternative way of viewing transducers, that might improve intuitions and serves as a justification for some naming later.
In particular, it clarifies the relationship between transducers and Bayes' rule, which plays a fundamental role when transducers are formulated this way.

In this section, given a function $q\colon X\to \distributions{Y}$ where $X$ and $Y$ are finite sets, we will write $q(y\parallel x)$ for the probability of the outcome $y$ according to the probability distribution $q(x)$, to stress the interpretation as conditional probability.

One can think of a transducer as something that takes in a sequence of inputs and stochastically generates a sequence of outputs, subject to the constraint that each element of the output sequence can only depend on the inputs that precede it in time. To do this we will show that the following definition is equivalent to definition \ref{def.transducer}.
\newcommand{\unrolledTransducer}[1]{p_{#1}}
\newcommand{\unrolledTransducers}{\inputSpace\btright\outputSpace}

\begin{definition}[Unrolled transducers]
  An \emph{unrolled transducer} from $\inputSpace$ to $\outputSpace$ consists of a family of functions $\unrolledTransducer{n} \colon \inputSpace^n \to \distributions{\outputSpace^{n+1}}$, one for each $n\in\N$.
  These must have the property that
  \begin{equation}
    \label{eq.unrolledTransducerCondition}
    \sum_{o \in \outputSpace} \unrolledTransducer{n+1}(\elementString{o} \cdot o \parallel \elementString{i} \cdot i)
    = \unrolledTransducer{n}( \elementString{o} \parallel \elementString{i}),
  \end{equation}
  for all sequences $\elementString{i} \cdot i \in \inputSpace^{n+1}$, $\elementString{o} \in \outputSpace^{n+1}$. We denote the set of unrolled transducers as $\unrolledTransducers$.
\end{definition}

In other words, for every finite sequence of inputs, an unrolled transducer specifies a probability distribution over finite sequences of outputs.
These output sequences are longer than the input sequences because we use the convention that a transducer gives its first output before it receives any input.

The condition in equation~\eqref{eq.unrolledTransducerCondition} serves two purposes:
firstly it says that these conditional probability distributions of output sequences must all be consistent with one another, and secondly it says that the marginal probability distribution over outputs up to time $n+1$ can only depend on the inputs up to time $n$.
This is the causality condition mentioned previously: outputs may depend on past inputs, but not on future ones.

Conceptually, an unrolled transducer specifies conditional probabilities of infinite output sequences given infinite input sequences, subject to a set of conditional independence relationships induced by the causality condition.
However, we avoid the use of measure theory by talking about probability distributions over finite sequences of every length instead of measures over infinite sequences.
This could be justified using a version of Kolmogorov's extension theorem, which says that knowing a probability distribution over every finite sequence of symbols (such that they agree on all the marginals) is equivalent to knowing a probability {measure} over the space of all infinite sequences.
We will never need to explicitly take such a step though, since by working coinductively we only ever need to consider finitely supported distributions over one symbol at a time.

Unrolled transducers are closely related to controlled stochastic processes.
Indeed, in the introduction we described unrolled transducers \emph{as} controlled stochastic processes.
This is justified because $\unrolledTransducers$ are equivalent to transducers $\genericTransducers$.
To show this, let us first define the unrolling of a transducer.

\newcommand{\unrolling}{u}
\begin{definition}
  The function $\unrolling \colon \genericTransducers \to \unrolledTransducers$ is defined inductively with the base
  \begin{equation}
    \unrolling_0\left( p, t \right)(o) = p(o),
  \end{equation}
  and the step
  \begin{equation}
    \unrolling_{n+1}\left( p, t \right)(o \cdot \elementString{o} \parallel i \cdot \elementString{i}) = p(o) \unrolling_n(t(i, o))(\elementString{o} \parallel \elementString{i}).
  \end{equation}
  As we will often do, we gloss over the transition sometimes not being defined, since this is only the case when the probability would be zero anyway.
\end{definition}

To see that this definition is correct, i.e.\ the resulting family of probabilities satisfies equation \eqref{eq.unrolledTransducerCondition}, we reason by induction. For the base we have
\begin{equation}
  \sum_{o' \in \outputSpace} \unrolling_1(p, t)(o \cdot o' \parallel i) = \sum_{o' \in \outputSpace} p(o) \probability_{t(i, o)}(o'),
\end{equation}
where the right hand side sums to $p(o) = \unrolling_0(p, t)(o)$ as desired. Similarly for the inductive step
\begin{equation}
  \sum_{o' \in \outputSpace} \unrolling_{n+1}(p, t)(o \cdot \elementString{o} \cdot o' \parallel i \cdot \elementString{i} \cdot i')
  = \sum_{o' \in \outputSpace} p(o) \unrolling_n(t(i, o))(\elementString{o} \cdot o' \parallel \elementString{i} \cdot i'),
\end{equation}
whence we can apply the inductive hypothesis
\begin{equation}
  p(o) \unrolling_{n-1}(t(i, o))(\elementString{o} \parallel \elementString{i}),
\end{equation}
which is exactly the definition of
\begin{equation}
  \unrolling_{n}(p, t)(o \cdot \elementString{o} \parallel i \cdot \elementString{i}),
\end{equation}
as desired.

It remains to define an inverse and show that it is such. To do this it's useful to first define a transition function for unrolled transducers.
\newcommand{\unrolledTransition}{s}
\begin{definition}
  The transition function for unrolled transducers
  \begin{equation*}
    \unrolledTransition \colon \prod_{p \colon \unrolledTransducers}\inputSpace \times \supp(p_0) \to \unrolledTransducers
  \end{equation*}
  is defined as
  \begin{equation}
    \unrolledTransition(p, i, o)_n(\elementString{o} \parallel \elementString{i}) = \frac{p_{n+1}(o \cdot \elementString{o} \parallel i \cdot \elementString{i})}{p_0(o)}.
  \end{equation}
\end{definition}
To see that the transition function for unrolled transducers actually returns an unrolled transducer, note that it just represents conditioning all the functions in the family on the first output being $o$ and the first input being $i$, which preserves condition \eqref{eq.unrolledTransducerCondition}.

Note that the transition function for unrolled transducers is nothing but Bayesian conditioning on an input-output pair. This fact will allow us to treat evolutions of transducers as Bayesian updates in the following. We just need to finish proving the equivalence.
\newcommand{\rolling}{r}
\begin{lemma}
  The function $\unrolling$ has an inverse defined coinductively as
  \begin{equation}
    \rolling( p ) = \left( p, (i, o) \mapsto \rolling( \unrolledTransition(p, i, o) \right),
  \end{equation}

  \begin{proof}
    We will first prove that the first composition $\unrolling(\rolling(p))$
    is the identity by induction, starting with
    \begin{equation}
      \unrolling_0(\rolling(p)) = p_0,
    \end{equation}
    and the step
    \begin{equation}
      \unrolling_{n+1}(\rolling(p)) = \unrolling_{n+1}\left(\left(p_0, (i, o) \mapsto \rolling(\unrolledTransition(p, i, o))\right)\right),
    \end{equation}
    which on specific arguments $o \cdot \elementString{o} \in \outputSpace^{n+2}$, $i \cdot \elementString{i} \in \inputSpace^{n + 1}$ is
    \begin{equation}
      p_0(o)\unrolling_n(\rolling(\unrolledTransition(p, i, o)))(\elementString{o} \parallel \elementString{i}),
    \end{equation}
    letting us apply the inductive hypothesis
    \begin{equation}
      p_0(o)\frac{p_{n+1}(o \cdot \elementString{o} \parallel i \cdot \elementString{i})}{p_0(o)} = p_{n+1}(o \cdot \elementString{o} \parallel i \cdot \elementString{i}),
    \end{equation}
    as required.

    Now let us examine the other composition
    \begin{equation}
      \left(\unrolling_0(p, t), (i, o) \mapsto \rolling(\unrolledTransition(\unrolling(p, t), i, o)))\right).
    \end{equation}
    Using the definition of $\unrolling_0$ and noting that the multiplication and division in definitions of $s$ and $\unrolling$ cancel out we get
    \begin{equation}
      \left(p, (i, o) \mapsto \rolling(\unrolling(t(i, o)))\right),
    \end{equation}
    so it remains to use the coinductive hypothesis to finish the proof.
  \end{proof}
\end{lemma}

Since this is the first time we are using this style of reasoning, let us
explain it in a bit more detail. We rely on a `coinductive hypothesis' that the relevant equality holds for any `smaller' or `deconstructed' object. In the case of the above proof we acquire such
an object in the form of $t(i, o)$ -- it is `smaller', intuitively, due
to us discarding some information, in the form of $p$, from the full
object. This principle can be applied not only to equality but also to properties known as `coinductively defined' properties; we make use of this in later sections. A reader interested in a more complete formal explanation of the above
can find it in \cite[Example 4.30]{CoinductionCzajka15}.

\subsection{Constrained transducers}
\label{sec.constrained}

In order to have a hope of being applied to practical
situations it is important to take into account the computational constrains of real world systems, as argued in
\cite[Section 4.1]{Williams21}. To this end we will sometimes consider constrained classes of transducers, which in principle are arbitrary subsets of all transducers $\constrainedTransducers \subseteq \genericTransducers$, but in practice will be somewhat less arbitrary. In this section we will describe some properties such classes can have and define some particularly useful ones.

The idea is that a constrained class of transducers represents the set of all transducers that respect some kind of practical constraint on computation, such as a limited memory.
In Section~\ref{sec.teleo} we will address constrained optimality by considering transducers that are optimal within such a class.

 A choice presents itself, however: if we want to use a constrained class as a reference class for optimality in this way, how should we consider it to change under evolution?
We should not in general say the same class is still the
  reference class after the evolution, since it might no longer contain the
  evolved transducer. We make a choice and say that after evolving a transducer
  $\transducer \colon \constrainedTransducers$ by a specific
  input-output pair $\left( i, o \right) \in \inputSpace \times \outputSpace$
  the proper reference class for the evolved transducer $\transducer \evolution
  \left( i, o \right)$ is $\constrainedTransducers \evolution \left( i, o
  \right)$. We define this for arbitrary trajectories.

  \begin{definition}
    \label{def.constrainedEvolution}
    Given a set of constrained transducers $\constrainedTransducers \subseteq
    \genericTransducers$ and a trajectory $\interleaved{i}{o}$, we define
    \begin{equation}
      \label{eq.constrainedEvolution}
      \constrainedTransducers \evolution \interleaved{i}{o} = \left\{ \transducer \colon \genericTransducers \mid \transducer = \transducer' \evolution \interleaved{i}{o} \textrm{ for some } \transducer' \colon \constrainedTransducers \right\}
    \end{equation}
    In other words it is the image of the evolution by this trajectory, but only
    for valid evolutions.
  \end{definition}

  We could have used a more
  general notion here, and many of the theorems would have still worked, but we
  have picked this one in the interest of keeping the paper focused. The
  definition we are using essentially means we are always comparing transducers
  to others with the same observable history.

\subsubsection{Splicing transducers}

In this section we introduce the operation of \emph{splicing} transducers along
a trajectory. We will later use the property of a constrained set of transducers
being closed under splicing in a specific way as an assumption in our main theorem.

\newcommand{\splicedAlong}[1]{\curlyvee_{#1}}

\begin{definition}[Splicing]
  Given a transducer $\transducer \colon \genericTransducers$ and
  a trajectory $\interleaved{ \elementString{i}}{ \elementString{o} } \in \left(
    \inputSpace \times \outputSpace \right)^*$ valid for evolving it, we define
  \emph{$\transducer$ spliced with $\transducer' \colon \genericTransducers$
    along $\interleaved{ \elementString{i}}{ \elementString{o} }$}, written $\pi\splicedAlong{\interleaved{\elementString{i}}{\elementString{o}}}\pi'$, inductively, with the base
  \begin{equation}
   \transducer \splicedAlong{\interleaved{ \emptyString }{ \emptyString }} \transducer' = \transducer'
  \end{equation}
  and the step
  \begin{equation}
   \transducer \splicedAlong{\interleaved{ i \cdot \elementString{i}}{ o \cdot \elementString{o} }} \transducer' =
    \left( \probability_{\transducer}, \left( j, q \right) \mapsto \begin{cases}
      \transducer \evolution \left( i, o \right) \splicedAlong{\interleaved{ \elementString{i}}{ \elementString{o} }} \transducer' & \textrm{ when } \left( j, q \right) = \left( i, o \right),\\
      \transducer \evolution \left( j, q \right) & \textrm{ otherwise,}
    \end{cases}
  \right).
  \end{equation}
\end{definition}

In other words the resulting transducer behaves exactly like $\transducer$,
except after the specified trajectory it starts behaving like $\transducer'$.

\begin{definition}[Closed under trajectory splicing]
 \label{def.cuts}
  We say that a set of constrained transducers $\constrainedTransducers\subseteq
  \genericTransducers$ is closed under trajectory splicing if for
  any transducers $\transducer, \transducer' \colon \constrainedTransducers$ and
  any trajectory $\interleaved{ \elementString{i}}{ \elementString{o} } \in \left(
    \inputSpace \times \outputSpace \right)^*$ valid for evolving them both, the
  transducer $\transducer \splicedAlong{\interleaved{ \elementString{i}}{
      \elementString{o} }} \left( \transducer' \evolution \interleaved{
      \elementString{i}}{ \elementString{o} } \right)$ is also in $\constrainedTransducers$.
\end{definition}

The above definition is straightforward, but it is not a coinductive definition.
However, being closed under trajectory splicing can also be defined in a
coinductive way, and we will show this is equivalent.

\begin{definition}
  \label{def.cutsCounductive}
  We say that a set of constrained transducers
  $\constrainedTransducers\subseteq \genericTransducers$ is closed under
  trajectory splicing if given any transducer $\transducer \colon
  \constrainedTransducers$, any input-output pair $\left( i, o \right)\in
  \inputSpace \times \outputSpace$ valid for evolving it and any transducer $\transducer' \colon \constrainedTransducers
  \evolution \left( i, o \right)$, the spliced transducer
  $\transducer \splicedAlong{\left( i, o \right)} \transducer'$ also belongs to
  $\constrainedTransducers$. In addition, we require that  $\constrainedTransducers
  \evolution \left( j, q \right)$ is  closed under trajectory splicing for
  all input-output pairs $\left( j, q \right) \in \inputSpace \times \outputSpace$.
\end{definition}

Note that the latter part of this definition makes it obvious that any
constrained class of transducers that is closed under trajectory splicing
remains so after an evolution by an arbitrary trajectory.

  \begin{proof}[Proof of definition equivalence]
    First note that being closed under trajectory splicing for the empty trajectory
    is trivial for any set, thus we will only consider nonempty trajectories.

  We start by proving that constrained transducer sets that have the property
  from Definition~\ref{def.cutsCounductive} are closed under trajectory
  splicing, inductively on the length of the trajectory. We get the base of the
  induction for free from the empty trajectory, so we only need to consider
  trajectories of the form $\interleaved{i \cdot \elementString{i}}{ o \cdot
  \elementString{o}}$. By applying the definitions of splicing and evolution a
  couple of times we get
  \begin{equation}
    \transducer \splicedAlong{\interleaved{i \cdot \elementString{i}}{ o \cdot \elementString{o}}} \left( \transducer' \evolution \interleaved{i \cdot \elementString{i}}{ o \cdot \elementString{o}} \right) =
    \transducer \splicedAlong{\left( i, o \right)} \left( \left( \transducer \evolution \left( i, o \right) \right) \splicedAlong{\interleaved{ \elementString{i}}{ \elementString{o}}} \left( \transducer' \evolution \left( i, o \right) \right) \evolution \interleaved{ \elementString{i}}{ \elementString{o}} \right).
  \end{equation}
  Since $\constrainedTransducers \evolution \left( i, o \right)$ is closed under
  trajectory splicing, the outer parentheses on the right hand side is in
  $\constrainedTransducers \evolution \left( i, o \right)$ by induction, which
  makes the entire expression be in $\constrainedTransducers$ by the first
  property from Definition \ref{def.cutsCounductive}, as required.

  For the converse note that every element of $\constrainedTransducers
  \evolution \left( i, o \right)$ is of the form $\transducer' \evolution \left(
  i, o \right)$, thus we get the first condition directly from $\transducer
  \splicedAlong{\left( i, o \right)} \left( \transducer' \evolution \left( i, o
  \right) \right)$ being in $\constrainedTransducers$. For the second condition
  consider a pair $\left( i, o \right)$, and the equation
  \begin{equation}
    \left( \transducer \splicedAlong{\interleaved{i \cdot \elementString{i}}{ o \cdot \elementString{o}}} \transducer' \evolution \interleaved{i \cdot \elementString{i}}{ o \cdot \elementString{o}} \right) \evolution \left( i, o \right) =
    \left( \transducer \evolution \left( i, o \right) \splicedAlong{\interleaved{\elementString{i}}{ \elementString{o}}} \left( \transducer' \evolution \left( i, o \right) \right) \evolution \interleaved{\elementString{i}}{ \elementString{o}} \right).
  \end{equation}
  The first parentheses on the left hand side are in $\constrainedTransducers$
  by assumption, so the entire expression is in $\constrainedTransducers
  \evolution \left( i, o \right)$. On the right hand side we can see arbitrary
  transducers from $\constrainedTransducers$ evolved by $\left( i, o \right)$,
  that is arbitrary transducers in $\constrainedTransducers \evolution \left( i,
  o \right)$. Since the expression is in $\constrainedTransducers \evolution
  \left( i, o \right)$ and the trajectories $\interleaved{\elementString{i}}{
  \elementString{o}}$ are arbitrary this is exactly being closed under
  trajectory splicing for $\constrainedTransducers \evolution \left( i, o
  \right)$, so this set satisfies the property from Definition
  \ref{def.cutsCounductive} by coinduction.
  \end{proof}

A similar reasoning to the last part of the above proof can be used to show that
if a class can be defined purely coinductively, i.e.\ by constructing its
elements so that they can evolve into an arbitrary member of this class, then it
is closed under trajectory splicing. This includes the class of all transducers,
although that also follows from the fact that splicing is well-defined.

  To illustrate this property we move on to some examples.

\begin{example}[i.i.d.\ transducers]
  \label{ex.constant}
  An \emph{independent and identically distributed transducer} or \emph{i.i.d.\ transducer} is one that can be defined as
  \begin{equation}
    \transducer = \left( p, \left( i, o \right) \mapsto \transducer \right),
  \end{equation}
  for some $p \in \distributions{\outputSpace}$. Note that while a single
  i.i.d.\ transducer is defined coinductively, the whole class cannot be.

  i.i.d.\ transducers are not closed under trajectory splicing if the support of
  $p$ has more than one element, since if we replaced the future transducer by
  another i.i.d.\ one the resulting transducer would no longer be i.i.d.
\end{example}
\newcommand{\deterministicTransducers}{D}
\begin{example}[Deterministic transducers]
  \label{ex.deterministic}
		\hspace{1em} A \emph{deterministic transducer} $\transducer \colon \deterministicTransducers \subset \genericTransducers$ is one for which the associated probability distribution is a one-point distribution and the transition function only takes values in further deterministic transducers $\deterministicTransducers$.

  This is a purely coinductive definition, so deterministic transducers are
  closed under trajectory splicing.

  We will also sometimes refer to intersections of this class with other classes by adding the adjective `deterministic' to the name of the given class.
\end{example}
More generally any constraint on probability distributions independent of any
transducer structure (e.g.\ minimal probability assigned to all possibilities,
uniform over a subset, etc.) gives rise to a class of transducers closed
under trajectory splicing. However, this is not a necessary condition, as the
following class demonstrates.
\begin{example}[Smearing transducers]
  A \emph{smearing transducer} of order $k \in \N$ is defined as one that has an associated probability with a support with at most $k$ elements, and its transition function takes values in smearing transducers of order $k + 1$.

  We assert that smearing transducers of order $k$ are closed under trajectory splicing, because only the length of the history impacts what transitions are possible.
\end{example}
\begin{example}[One-flip transducers]
  \label{ex.oneFlip}
  A \emph{one-flip transducer} is one for which the associated probability distribution either is a one-point distribution and the transition function takes values within one-flip transducers, or an arbitrary distribution and the transition function takes values in deterministic transducers. Thus these transducers can have at most one associated random variable not be deterministic (they have `one coin flip', hence the name).

  One-flip transducers are not closed under trajectory splicing, because once they use their coin flip they cannot have a transition lead to a one-flip transducer that has not yet used the coinflip. They will also be an important example for another reason later.
\end{example}

\subsubsection{Unifilar machines}
\label{sec.unifilarMachines}

\newcommand{\unifilarStateSet}{X}
\newcommand{\unifilarState}{x}
\newcommand{\unifilarOutputSet}{\distributions{\outputSpace}}
\newcommand{\unifilarOutput}[1]{\varphi_{#1}}
\newcommand{\explicitUnifilarTransitionSet}[2]{\inputSpace \times \supp\left( #1 \right) \to #2}
\newcommand{\unifilarTransitionSet}[1]{\explicitUnifilarTransitionSet{\unifilarOutput{}}{#1}}
\newcommand{\unifilarTransition}[1]{\theta_{#1}}
\newcommand{\unifilarMachine}{\left( \unifilarOutput{}, \unifilarTransition{} \right)}
\newcommand{\transducerFunctor}[1]{\sum_{\unifilarOutput{} \colon \unifilarOutputSet} \unifilarTransitionSet{#1} }
\newcommand{\unifilarMachineSetDefinition}[1]{#1 \to \transducerFunctor{#1}}
\newcommand{\unifilarMachineSet}[1]{ U_{#1} }
In this and the following sections we introduce several formalisms for modelling state-based dynamics, which can easily be translated into transducers. We start with unifilar machines, as they are used as an intermediate step in later sections.

\begin{definition}[Unifilar machines]
  \label{def.unifilarMachines}
  Let $\unifilarStateSet$ be an arbitrary set called the state set. We will call $\unifilarMachine \colon \unifilarMachineSetDefinition{\unifilarStateSet}$ a \emph{unifilar machine}. Given a \emph{state} $\unifilarState \colon \unifilarStateSet$  we will write $\left( \unifilarOutput{\unifilarState}, \unifilarTransition{\unifilarState} \right) = \unifilarMachine\left( x \right)$ and call it a pointed unifilar machine. The elements of the pair $\unifilarMachine$ will be called its \emph{output function}, and \emph{transition function} respectively. We denote the whole set of unifilar machines over $\unifilarStateSet$ as $\unifilarMachineSet{\unifilarStateSet}$.
\end{definition}

Unifilar machines take their name from the similar concept of unifilarity defined in the context of $\varepsilon$-transducers in \cite{barnett2015-epsilon}.

A reader familiar with more explicit treatments of coinduction may notice that a unifilar machine is a coalgebra of the functor $\transducerFunctor{{-}}$, while, as already mentioned, the set of all transducers, $\genericTransducers$, is the carrier of the final coalgebra of this functor. Our coinductive definition is equivalent to specifying the final coalgebra of this functor.
This will not make any significant difference below, as we explicitly define a map into transducers anyway.

\newcommand{\unifilarIntoTransducers}{s}
\begin{definition}
  \label{def.unifilarMapping}
  Define a mapping $\unifilarIntoTransducers \colon \unifilarStateSet \times \unifilarMachineSet{\unifilarStateSet} \to \genericTransducers$ from pointed unifilar machines to transducers corecursively as
  \begin{equation}
    \unifilarIntoTransducers\left(
      \unifilarState,
      \unifilarMachine
    \right) = \left(
      \unifilarOutput{\unifilarState},
      \left( i, o \right) \mapsto \unifilarIntoTransducers\left(
        \unifilarTransition{\unifilarState}\left( i, o \right),
        \unifilarMachine
      \right)
    \right).
  \end{equation}
\end{definition}

\subsubsection{Unifilar finite state machines}
\newcommand{\unifilarFiniteStateTransducers}[1]{\mathrm{UFS}_{#1}}
We will sometimes want to consider certain subsets of transducers, corresponding to transducers that can be constructed in a particular way. In this section we will see our first example. We consider a special case of unifilar machines, which represents a formalism for limited memory (but unlimited time) computation. We then discuss the transducers that can be constructed from these machines.
\begin{definition}[Unifilar finite state machines]
  Given a nonempty finite state set of size $n$ $\unifilarStateSet$, we will call $\unifilarMachine \in \unifilarMachineSet{\unifilarStateSet}$ a \emph{unifilar finite state (UFS) machine}.
\end{definition}
We interpret this state set as all the possible states of memory, and output and transition functions as code that maps a memory state to a probability distribution over outputs and a memory state plus an input and output into a new memory state respectively. In this abstraction the computational process has access to a randomness source.

Unifilar finite state machines are roughly equivalent to deterministic finite state (Moore) machines, although some of the considered sets are infinite in contrast to the classical definition.

We can now map these unifilar transducers into transducers using $\unifilarIntoTransducers$ defined in the previous section. As one might expect the state set will become irrelevant through this mapping, so we will encounter the problem that there are multiple distinct UFS machines that give rise to the same transducers (e.g.\ by composing the machine with a bijection to another set), so we will have to keep that in mind when discussing examples.

\begin{definition}[Unifilar finite state transducers]
  We call the image $\unifilarFiniteStateTransducers{n} = \unifilarIntoTransducers\left( \unifilarStateSet \times \unifilarMachineSet{\unifilarStateSet} \right)$ \emph{unifilar finite state (UFS) transducers}.
\end{definition}

It is easy to see that for all $n \in \Z_+$ $\unifilarFiniteStateTransducers{n} \subseteq \unifilarFiniteStateTransducers{n+1}$. It suffices to notice that the subset of UFS machines that identifies two elements of the memory set (i.e the machine behaves identically regardless of which of these elements is passed to it) gives rise to the UFS transducers with one less possible memory state.

It is similarly easy to see that for $n = 1$ UFS transducers are exactly i.i.d.
transducers, so they are not closed under trajectory splicing if the output set
has more than one element. More generally no UFS transducers are closed under
trajectory splicing if that's the case.

\subsubsection{Stochastic Moore machines}
\label{sec.stochasticMooreMachines}

\newcommand{\mooreStateSet}{Y}
\newcommand{\mooreState}{y}
\newcommand{\mooreStateMixture}{\psi}
\newcommand{\mooreOutputSet}[1]{\Phi_{#1}}
\newcommand{\mooreOutput}{\varphi}
\newcommand{\mooreTransitionSet}[1]{\Theta_{#1}}
\newcommand{\mooreTransition}{\theta}
\newcommand{\mooreMachineSet}[1]{\distributions{#1} \times \mooreOutputSet{#1} \times \mooreTransitionSet{#1} }
In this section we introduce yet another notion of state-based dynamics. We won't be using any object constructed here in the rest of the paper, but they serve as important motivation and context for interpreting the main result.

This is the only section explicitly using measure theory.
For the most part, a reader unfamiliar with measure theory may ignore the measure theoretic details.
In previous sections we used $\distributions{A}$ to mean the set of probability distributions over a finite set $A$.
In this section we extend this notation, so that $\distributions{\mooreStateSet}$ means the space of probability measures over a measurable space $\mooreStateSet$, where $\distributions{\mooreStateSet}$ itself is to be viewed as a measurable space, equipped with the smallest $\sigma$-algebra such that for every event in the $\sigma$-algebra of $S$, the map $\distributions{\mooreStateSet}\to [0,1]$ given by $\mu\mapsto \mu(S)$ is measurable.
In category theory terms this means that in this section $P$ stands for
the functor underlying the Giry monad \cite{Giry82, Fritz2020}.
When $X$ is a finite set, $\distributions{X}$ still amounts to the set of probability distributions over $X$.
We will continue to assume $I$ and $O$ are finite sets.

Recall that for measurable spaces $X$ and $Y$, a \emph{Markov kernel} from $X$ to $Y$ is a measurable function~$X\to \distributions{Y}$.  Given a Markov kernel $q\colon X\to \distributions{Y}$ in which $Y$ is finite, we will write $q(y\parallel x)$ for the probability of the outcome $y$ according to the probability distribution $q(x)$, as in section \ref{sec.unrolling}.

We first define a basic notion of a machine with an internal state:
\newcommand{\mooreSet}{M_Y}
\begin{definition}[Stochastic Moore machines]
  \label{def.stochasticMooreMachine}
  A stochastic Moore machine consists of a measurable space $\mooreStateSet$ called the \emph{state space}, together with a Markov kernel $\mooreOutput\colon \mooreStateSet\to \distributions{\outputSpace}$ called the \emph{output kernel}, a Markov kernel $\mooreTransition\colon \mooreStateSet \times \inputSpace \to \distributions{\mooreStateSet}$ called the \emph{transition kernel} and a probability measure $\mooreStateMixture\in \distributions{\mooreStateSet}$ called the \emph{initial distribution}. We call the triple $\left(\mooreStateMixture, \mooreOutput, \mooreTransition \right)$ a \emph{stochastic Moore machine} over $\mooreStateSet$. We denote the set of stochastic Moore machines over $\mooreStateSet$ as $\mooreSet$.
\end{definition}

The idea is that a stochastic Moore machine has an internal state, which is an element of $\mooreStateSet$.
It produces an output stochastically according to its output kernel.
Then, after being given an input, it transitions stochastically to a new state, according to its transition function.
In general the state of a stochastic Moore machine might not be known.
(And even if it is initially known, it will generally be in an unknown state after a transition, assuming the observer can't see inside the machine.)
We define a stochastic Moore machine to be equipped with an initial probability measure over its state space, to be thought of as a prior.

In contrast to unifilar machines, which can only introduce nondeterminism via the output kernel, stochastic Moore machines can have nondeterministic transitions that cannot depend on the output.
Despite this difference, given a stochastic Moore machine there is a way to define a unifilar machine with the same behaviour:
\newcommand{\mooreToUnifilar}{z}
\begin{definition}
  \label{def.mooreToUnifilar}
  Given a state space $\mooreStateSet$, we can map stochastic Moore machines over $\mooreStateSet$ into pointed unifilar machines over $\distributions{\mooreStateSet}$. We denote this map
  \begin{equation*}
    \mooreToUnifilar\colon \mooreSet \to \left( \distributions{\mooreStateSet} \times \unifilarMachineSet{\distributions{\mooreStateSet}} \right).
  \end{equation*}
  It is defined on the components as
  \begin{equation}
    \mooreToUnifilar\left( \mooreStateMixture \right) = \mooreStateMixture,
  \end{equation}
  on the state mixture, with the following operation on the output kernels
  \begin{equation}
    \label{eq.stateOutput}
    \mooreToUnifilar\left( \mooreOutput \right)\left( \mooreStateMixture \right) =
    \int_{\mooreStateSet} \mooreOutput\left( \mooreState \right) \mooreStateMixture\left( \mooreState \right) \mathrm{d}\mooreState,
  \end{equation}
  and the somewhat more complex operation on transition kernels
  \begin{equation}
    \label{eq.stateTransition}
    \mooreToUnifilar\left( \mooreTransition \right)\left( \mooreStateMixture, i, o \right) =
    \frac{1}{\mooreToUnifilar\left( \mooreOutput \right)\left( o \parallel \mooreStateMixture \right)}
    \int_{\mooreStateSet} \mooreTransition\left( \mooreState, i \right) \mooreOutput\left( o \parallel \mooreState \right) \mooreStateMixture\left( \mooreState \right) \mathrm{d}\mooreState.
  \end{equation}
\end{definition}
A remark is in order on the meaning of this map.
On the one hand, intuitively, the pointed unifilar machine $\mooreToUnifilar(\mooreStateMixture, \mooreOutput, \mooreTransition)$ `behaves the same' as the stochastic Moore machine $(\mooreStateMixture, \mooreOutput, \mooreTransition)$, but on the other hand it is a different type of machine and has a different state space.
The states of the unifilar machine consist of probability distributions over $\mooreStateSet$ instead of elements of $\mooreStateSet$.
One way to think of this is that these distributions represent an agent's state of knowledge about the state of the corresponding Moore machine.
Initially this state of knowledge is given by the stochastic Moore machine's initial distribution.
Then equation \eqref{eq.stateOutput} can be seen as a prediction of the Moore machine's next output, and \eqref{eq.stateTransition} can be seen as a Bayesian update, producing a posterior over the Moore machine's states, conditioned on the given input and the observed output.
\begin{remark}
  \label{rem.mooreUpdateIsFiltering}
  This updating of a probability distribution over $\mooreStateSet$ may
  be seen as an instance of Bayesian filtering with an additional input \cite{Virgo2023unifilar}.
  This perspective will be further justified in Section \ref{sec.evolutionIsFiltering}.
\end{remark}
\newcommand{\mooreIntoTransducers}{\unifilarIntoTransducers \circ \mooreToUnifilar}
Now we can use the composition $\mooreIntoTransducers$ to map stochastic Moore machines into transducers. As one would expect, this mapping respects mixtures in the first argument, we will show that in two steps.
\begin{lemma}
  \label{lem.unifilarRespectsMixtures}
  \newcommand{\kthOutputProbability}[1]{ \unifilarOutput{\mooreStateMixture_{#1}}\left( o \right)}
  \newcommand{\kthUnifilar}[1]{\left( \mooreStateMixture_{#1}, \unifilarMachine \right)}
  \newcommand{\kthUnifilarIntoTransducers}[1]{\unifilarIntoTransducers\kthUnifilar{#1}}
  Suppose we are given a machine $\unifilarMachine \colon \unifilarMachineSet{\distributions{\mooreStateSet}}$ for which the output function respects mixtures and the transition function satisifes
  \begin{equation}
    \label{eq.transitionReactsToMixture}
    \unifilarTransition{\mixture{\alpha_}{\mooreStateMixture_}}\left( i, o \right) =
    \sum_{k = 1}^{\infty}
    \frac{\alpha_{k} \kthOutputProbability{k}}{\explicitMixture{l}{\alpha_}{\kthOutputProbability}}
    \unifilarTransition{\mooreStateMixture_{k}}\left( i, o \right),
  \end{equation}
  for any numbers $\alpha_{k} \geq 0$ such that $\sum_{k = 0}^{\infty} \alpha_{k} = 1$. Then for any such numbers $\unifilarIntoTransducers$ respects mixtures of (distributions of) states for an unifilar machine with these output and transition functions, that is
  \begin{equation}
    \label{eq.unifilarRespectsMixtures}
    \unifilarIntoTransducers\left( \mixture{\alpha_}{\mooreStateMixture_}, \unifilarMachine \right) =
    \mixture{\alpha_}{\kthUnifilarIntoTransducers}.
  \end{equation}

  \begin{proof}
    Starting with the left hand side and expanding the definition of $\unifilarIntoTransducers$
    \begin{equation}
      \unifilarIntoTransducers\left( \mixture{\alpha_}{\mooreStateMixture_}, \unifilarMachine \right) =
      \left(
        \unifilarOutput{\mixture{\alpha_}{\mooreStateMixture_}},
        \left( i, o \right) \mapsto
        \unifilarIntoTransducers\left(
          \unifilarTransition{\mixture{\alpha_}{\mooreStateMixture_}}\left(
            i,
            o
          \right),
          \unifilarMachine
        \right)
      \right).
    \end{equation}
    \newcommand{\kthOutput}[1]{ \unifilarOutput{\mooreStateMixture_{#1}}}
    Applying the fact that $\unifilarOutput{}$ respects mixtures and the property from equation \eqref{eq.transitionReactsToMixture} we get
    \begin{equation}
      \left(
        \mixture{\alpha_}{\kthOutput},
        \left( i, o \right) \mapsto
        \unifilarIntoTransducers\left(
          \sum_{k = 1}^{\infty}
          \frac{\alpha_{k} \kthOutputProbability{k}}{\explicitMixture{l}{\alpha_}{\kthOutputProbability}}
          \unifilarTransition{\mooreStateMixture_{k}}\left( i, o \right),
          \unifilarMachine
        \right)
      \right).
    \end{equation}
    Using the coinductive hypothesis we can extract the mixture from the transition
    \begin{equation}
      \left(
        \mixture{\alpha_}{\kthOutput},
        \left( i, o \right) \mapsto
        \sum_{k = 1}^{\infty}
        \frac{\alpha_{k} \kthOutputProbability{k}}{\explicitMixture{l}{\alpha_}{\kthOutputProbability}}
        \unifilarIntoTransducers\left(
          \unifilarTransition{\mooreStateMixture_{k}}\left( i, o \right),
          \unifilarMachine
        \right)
      \right).
    \end{equation}
    This is exactly the definition of a mixture of transducers, giving us
    \begin{equation}
      \sum_{k = 1}^{\infty} \alpha_{k}
      \left(
        \kthOutput{k},
        \left( i, o \right) \mapsto
        \unifilarIntoTransducers\left(
          \unifilarTransition{\mooreStateMixture_{k}}\left( i, o \right),
          \unifilarMachine
        \right)
      \right).
    \end{equation}
    Collapsing the definition of $\unifilarIntoTransducers$, gets us the desired
    \begin{equation}
      \mixture{\alpha_}{\kthUnifilarIntoTransducers}.
    \end{equation}
  \end{proof}
\end{lemma}
\begin{lemma}
  \label{lem.stateMappingRespectsMixtures}
  \newcommand{\kthMooreIntoTransducers}[1]{\mooreIntoTransducers\left( \mooreStateMixture_{#1}, \mooreOutput, \mooreTransition \right)}
  Given state distributions $\mooreStateMixture_{k} \colon \distributions{\unifilarStateSet}$ and numbers $\alpha_{k} \geq 0$, such that $\sum_{k = 0}^{\infty} \alpha_{k} = 1$ we have
  \begin{equation}
    \mooreIntoTransducers\left(
      \mixture{\alpha_}{\mooreStateMixture_},
      \mooreOutput,
      \mooreTransition
    \right) =
    \mixture{\alpha_}{\kthMooreIntoTransducers}.
  \end{equation}

  \begin{proof}
    \newcommand{\kthIntegralByStateDistributionOf}[2]{\int_{\mooreStateSet} #1 \mooreStateMixture_{#2}\left( \mooreState \right) \mathrm{d}\mooreState}
    \newcommand{\kthMappedMooreOutput}[1]{\mooreToUnifilar\left( \mooreOutput \right)\left( \mooreStateMixture_{#1} \right)}
    We will prove that $\mooreToUnifilar\left( \mooreStateMixture, \mooreOutput, \mooreTransition \right)$ satisfies the conditions of Lemma \ref{lem.unifilarRespectsMixtures} and, noting that it is the identity on the first component, an application of that lemma finishes the proof.

    Let us start with the output function respecting mixtures
    \begin{multline}
      \mooreToUnifilar\left( \mooreOutput \right)\left( \mixture{\alpha_}{\mooreStateMixture_} \right) =
      \int_{\mooreStateSet} \mooreOutput\left( \mooreState \right) \mixture{\alpha_}{\mooreStateMixture_}\left( \mooreState \right) \mathrm{d}\mooreState \\=
      \mixture{\alpha_}{\kthIntegralByStateDistributionOf{\mooreOutput\left( \mooreState \right)}} =
      \mixture{\alpha_}{\kthMappedMooreOutput}.
    \end{multline}

    \newcommand{\kthMappedMooreTransition}[1]{\mooreToUnifilar\left( \mooreTransition \right)\left( \mooreStateMixture_{#1}, i, o \right)}
    \newcommand{\kthMappedMooreOutputAtO}[1]{\mooreToUnifilar\left( \mooreOutput \right)\left( o \parallel \mooreStateMixture_{#1} \right)}
    Having this we can compute
    \begin{multline}
      \mooreToUnifilar\left( \mooreTransition \right)\left( \mixture{\alpha_}{\mooreStateMixture_}, i, o \right) =\\
      \frac{1}{\explicitMixture{l}{\alpha_}{\kthMappedMooreOutputAtO}}
      \int_{\mooreStateSet} \mooreTransition\left( \mooreState, i \right) \mooreOutput\left( o \parallel \mooreState \right) \mixture{\alpha_}{\mooreStateMixture_}\left( \mooreState \right) \mathrm{d}\mooreState
    \end{multline}
    writing $1 = \frac{\kthMappedMooreOutputAtO{k}}{\kthMappedMooreOutputAtO{k}}$ and performing some simple operations on sums and integrals we get
    \begin{multline}
      \sum_{k = 0}^{\infty}
      \frac{\alpha_{k} \kthMappedMooreOutputAtO{k}}{\explicitMixture{l}{\alpha_}{\kthMappedMooreOutputAtO}}
      \frac{1}{\kthMappedMooreOutputAtO{k}}
      \int_{\mooreStateSet} \mooreTransition\left( \mooreState, i \right) \mooreOutput\left( o \parallel \mooreState \right) \mooreStateMixture_{k}\left( \mooreState \right) \mathrm{d}\mooreState \\=
      \sum_{k = 0}^{\infty}
      \frac{\alpha_{k} \kthMappedMooreOutputAtO{k}}{\explicitMixture{l}{\alpha_}{\kthMappedMooreOutputAtO}}
      \kthMappedMooreTransition{k}.
    \end{multline}
    It remains to apply Lemma \ref{lem.unifilarRespectsMixtures} and we get the desired result.
  \end{proof}
\end{lemma}
We perform the two steps separately to stress that only the first proof uses coinduction. The second could not have used it directly, as $\mooreToUnifilar$ is not coinductively defined.

We conjecture that there is a $\sigma$-algebra on $\genericTransducers$ that, among other things, would allow one to define maps in the other direction, from transducers (or subsets of transducers) to stochastic Moore machines. This in turn would allow some of the machinery of \cite{Virgo2023unifilar} to be applied in our context. The existence of such a $\sigma$-algebra is not obvious, as it requires several functions, both from and into transducers to be measurable. A suitable $\sigma$-algebra on the set of transducers would also allow reasoning about uncountable mixtures of transducers.

\section{Teleo-Environments}
\label{sec.teleo}

We finally define the main objects we will be investigating.
\begin{definition}[Teleo-environment]
  \label{def.teleoEnv}
  Let $\stateSpace$ and $\actionSpace$ be nonempty finite sets called the \emph{state space} and \emph{action space} respectively. Furthermore let $\telosSpace = \left\{ \noSuccess, \success  \right\}$, called the \emph{telos channel}, with the contents of the set called \emph{nothing} and \emph{success} respectively. We then call $\environment \colon \environments$ a \emph{teleo-environment} and $\policy \colon \policies$ a \emph{policy}.
\end{definition}

One should think of `success' and `nothing' as follows: an agent's goal in interacting with a teleo-environment is to maximise its probability of achieving success \emph{at least once.}
So $\noSuccess$ does not represent failure, but only that success has not been achieved on this particular time step.
The agent does not receive the success signal as an input, so in general the agent will not know whether success has been achieved yet or not; this fact will turn out to be important.

We remark that instead of maximising the probability of success, one could define a similar set-up where an agent tries to minimise a probability of failure; we would expect similar results to hold in that case.
One could also consider an agent maximising a suitably bounded reward function, perhaps with exponential discounting, which is the more familiar set-up associated with POMDP tasks.
However, this has an important difference, in that success (or failure) probabilities combine multiplicatively rather than additively on successive time steps, and for this reason we would expect the resulting formalism to be somewhat different.
We will not consider either of these cases further.

There is a natural way to connect a policy and an environment together, getting a transducer that does not require inputs, which we will call a coupled system.

\subsection{Coupling}

\newcommand{\coupledTransducers}{\transducers{\{\star\}}{\stateSpace \times \telosSpace \times \actionSpace }}
\newcommand{\coupling}{c}
\begin{definition}[Coupling]
  \label{def.coupling}
  We define the \emph{coupling} function
  \begin{equation}
      \coupling \colon \left( \policies \right) \times \left( \environments \right) \to \coupledTransducers
  \end{equation}
  by setting the probability distribution to the product distribution
  \begin{equation}
    \probability_{\coupling\left( \policy, \environment \right)}\left( s, t, a \right) = \probability_{\policy}\left( a \right) \probability_{\environment}\left( s, t \right),
  \end{equation}
  and the transition function to a corecursive application of coupling to the transitions of the arguments
  \begin{equation}
    \transition_{\coupling\left( \policy, \environment \right)}\left( \star, \left( s, t, a \right) \right) = \coupling\left( \transition_{\policy}\left( s, a \right), \transition_{\environment}\left( a, \left( s, t \right) \right) \right).
  \end{equation}
\end{definition}
The $\star$ here is a dummy, and only possible, input to conform with the formal definition of a transducer.

The coupling operation will be crucial for our investigations, even though we will rarely refer to it explicitly. We will prove one important property it has though.
\begin{lemma}[Coupling respects mixtures]
  \label{lem.couplingOfMixtures}
  Let $\policy_{k} \colon \policies$ be policies, $\environment_{k} \colon \environments$ be environments, and $\alpha_{k}, \beta_{k} \geq 0$. The coupling function respects mixtures, that is
  \begin{equation}
    \coupling\left( \mixture{\alpha_}{\policy_}, \mixture{\beta_}{\environment_} \right) =
    \sum_{k = 1}^{\infty} \sum_{l = 1}^{\infty} \alpha_{k}\beta_{l} \coupling\left( \transducer_{k}, \environment_{l} \right).
  \end{equation}

  \begin{proof}
    Let us start with the probability
    \begin{equation}
      \probability_{\coupling\left( \mixture{\alpha_}{\policy_}, \mixture{\beta_}{\environment_} \right)}\left( s, t, a \right) = \probability_{\mixture{\alpha_}{\policy_}}\left( a \right) \probability_{\mixture{\beta_}{\environment_}} \left( s, t \right).
    \end{equation}
    Applying the definition of a mixture this becomes
    \begin{equation}
      \mixture{\alpha_}{\kthProbabilityAt{\policy}{a}} \mixture{\beta_}{\kthProbabilityAt{\environment}{s, t}} =
      \sum_{k = 1}^{\infty}\sum_{l = 1}^{\infty} \alpha_k \beta_l \probability_{\policy_k}\left( a \right) \probability_{\environment_l}\left( s, t \right),
    \end{equation}
    which after rewriting using the definition of coupling gives us exactly the definition of a probability of a mixture of couplings:
    \begin{equation}
      \sum_{k = 1}^{\infty}\sum_{l = 1}^{\infty} \alpha_{k} \beta_{l}\probability_{\coupling\left( \policy_k, \environment_{l} \right)}\left( s, t, a \right).
    \end{equation}

    Now the transition function is given by
    \begin{multline}
      \transition_{\coupling\left( \mixture{\alpha_}{\policy_}, \mixture{\beta_}{\environment_} \right)}\left( \star, \left( s, t, a \right) \right) =\\ \coupling\left( \transition_{\mixture{\alpha_}{\policy_}}\left( s, a \right), \transition_{\mixture{\beta_}{\environment_}}\left( a, \left( s, t \right) \right) \right).
    \end{multline}
    Applying the definition of a mixture it becomes
    \begin{equation}
      \coupling\left(
        \mixture{\kthCoefficientAt{\alpha}{\policy}{a}}{\kthTransitionAt{\policy}{s}{a}},
        \mixture{\kthCoefficientAt{\beta}{\environment}{s, t}}{\kthTransitionAt{\environment}{a}{s, t}}
      \right).
    \end{equation}
    Using the coinductive hypothesis and performing simple operations on sums it becomes
    \begin{equation}
      \sum_{k = 1}^{\infty}\sum_{l = 1}^{\infty} \frac{\alpha_k \beta_l \probability_{\policy_{k}}\left( a \right) \probability_{\environment_{l}}\left( s, t \right)}{\sum_{k' = 1}^{\infty} \sum_{l' = 1}^{\infty} \alpha_{k'} \beta_{l'} \probability_{\policy_{k'}}\left( a \right) \probability_{\environment_{l'}}\left( s, t \right)}
      \coupling\left( \transition_{\policy_{k}}\left( s, a \right), \transition_{\environment_{l}}\left( a, s, t \right) \right)
    \end{equation}
    and after rewriting using the definition of coupling it gives us exactly the definition of a transition of a mixture of couplings:
    \begin{equation}
      \sum_{k = 1}^{\infty}\sum_{l = 1}^{\infty} \frac{\alpha_k \beta_l \probability_{\coupling\left( \policy_{k}, \environment_{l} \right)}\left( s, t, a \right)}{\sum_{k' = 1}^{\infty} \sum_{l' = 1}^{\infty} \alpha_{k'} \beta_{l'} \probability_{\coupling\left( \policy_{k'}, \environment_{l'} \right)}\left( s, t, a \right)}
      \transition_{\coupling\left( \policy_{k}, \environment_{l} \right)}\left( \star, \left( s, t, a \right) \right).
    \end{equation}
  \end{proof}
\end{lemma}
\subsubsection{Coupled evolutions as filtering}
\label{sec.evolutionIsFiltering}
The term `Bayesian filtering' is used to describe a technique of performing Bayesian inference sequentially, to track a variable which changes over time. In terms of random variables: upon receiving data $D_n$ at time step $n$, with a prior $P(X_n \mid D_{<n})$, filtering produces a posterior $P(X_{n+1} \mid D_{<n+1})$. Note that the posterior is over a different variable $X_{n+1}$ than the prior, which is over $X_n$. In this section we will explain why the evolution operator on coupled transducers can be understood in terms of filtering.

Since a coupling has trivial input, we can treat its unrolling as a probability space over output sequences, allowing us to refer to random variables within that space. We will write $P_{\coupling\left( \policy, \environment\right)}(S_0, \cdots S_{n}, T_0, \cdots, T_n, A_0, \cdots, A_n)$ for the probability space induced by the distribution $\unrolling_n\left(\coupling\left( \policy, \environment\right)\right)$. Then the evolution on couplings can be understood in terms of their unrollings as follows. For any $n, m \in \mathbb{N}$, any sequences of state data $\elementString{s} \in \stateSpace^n, \elementString{s'} \in \stateSpace^m$, any sequences of telos data $\elementString{t} \in \telosSpace^n, \elementString{t'} \in \telosSpace^m$, any sequences of action data $\elementString{a} \in \actionSpace^n, \elementString{a'} \in \actionSpace^m$ we have:
\begin{multline}
  P_{\coupling\left(\policy \evolution \interleaved{\elementString{s}}{\elementString{a}}, \environment \evolution \interleaved{\elementString{a}}{\interleaved{\elementString{s}}{\elementString{t}}} \right)}(S_{1 \cdots  m} = \elementString{s'},T_{1 \cdots  m} = \elementString{t'},A_{1 \cdots  m} = \elementString{a'}) =\\
  P_{\coupling\left(\policy, \environment\right)}(S_{n+1 \cdots n+m} = \elementString{s'},T_{n+1 \cdots n+m} = \elementString{t'},A_{n+1 \cdots n+m} = \elementString{a'} \\ \mid S_{1 \cdots n} = \mathbf{s},T_{1 \cdots n} = \mathbf{t},A_{1 \cdots n} = \mathbf{a})
\end{multline}
In other words, the first $m$ steps of the unrolling of the coupling of
evolutions
\begin{equation}
  \coupling\left(\policy \evolution \interleaved{\elementString{s}}{\elementString{a}}, \environment \evolution \interleaved{\elementString{a}}{\interleaved{\elementString{s}}{\elementString{t}}} \right)
\end{equation}
behave like the steps $n+1 \cdots n+m$ of the unrolling of the coupling $\coupling\left(\policy, \environment\right)$, conditioned on the first $n$ outputs being equal to the evolution trajectory. In this sense, the evolution operator can be seen as an analogue of Bayesian filtering.

Recall that Remark \ref{rem.mooreUpdateIsFiltering} pointed out the similarity between Bayesian filtering and the operation on Moore machines which corresponded to transition. Since evolution is just repeated transition we now see that the comparison was not coincidental -- in the case of coupled systems this operation maps into filtering for unrollings.

\subsection{Success}

We now proceed to defining the probability of success. Note that we only care about encountering success at least once, multiple occurences are irrelevant.

We first define some helper objects and functions which will be crucial in proofs.
\newcommand{\coupledTransducer}{\chi}
\newcommand{\boundedSumSequences}[1]{\mathrm{BSS}\left( #1 \right)}
\begin{definition}[Bounded sum sequences]
  \label{def.boundedSumSequences}
  We define the family \emph{bounded sum sequences} $\boundedSumSequences{r_{0}}$ over the interval $\left[ 0, 1 \right]$ coinductively using the single constructor taking $r \in \left[ 0, r_0 \right]$ and $t \colon \boundedSumSequences{r_{0} - r}$ written as $\left( r, t \right)$.
\end{definition}
The idea is that $\boundedSumSequences{r_{0}}$ represents the set of sequences of numbers in $[0,1]$ that sum to at most $r_0$. Consequently when the first element is removed the remainder of the sequence sums to at most $r_0-r$, i.e.\ it is an element of $\boundedSumSequences{r_{0} - r}$.

There are natural multiplication and summing operations on bounded sum sequences.
\begin{definition}
  \label{def.infraBoundedSumSequences}
  Given numbers $r_{0}, s \in \left[ 0, 1 \right]$ we define multiplication
  \begin{equation}
      s\cdot \colon \boundedSumSequences{r_{0}} \to \boundedSumSequences{sr_{0}}
  \end{equation} corecursively as
  \begin{equation}
    s\left( r, t \right) = \left( s r,  s t \right).
  \end{equation}

  Given numbers $s, s_{k} \in \left[ 0, 1 \right]$ such that $\sum_{i = 0}^{\infty} s_k = s$ and a sequence of bounded sum sequences $\left( r_{k}, t_{k} \right) \colon \boundedSumSequences{s_{k}}$ we define summing corecursively as
  \begin{equation}
    \sum_{k = 0}^{\infty} \left( r_{k}, t_{k} \right) = \left( \sum_{k = 0}^{\infty} r_{k}, \sum_{k = 0}^{\infty} t_{k} \right),
  \end{equation}
  which belongs to $\boundedSumSequences{s}$.
\end{definition}
We can use bounded sum sequences to represent the probability that a specific coupled system outputs its first success exactly at a given step.
\newcommand{\successSequence}{\sigma}
\begin{definition}[Success sequence]
  \label{def.successSequence}
  We define the \emph{success sequence} $\successSequence \colon \coupledTransducers \to \boundedSumSequences{1}$ corecursively as
  \begin{multline}
    \successSequence\left( \coupledTransducer \right) = \left(
      \sum_{s \in \stateSpace}\sum_{a \in \actionSpace} \probability_{\coupledTransducer}\left( s, \success, a \right),\right.\\
    \left. \left( \sum_{s \in \stateSpace}\sum_{a \in \actionSpace} \probability_{\coupledTransducer}\left( s, \noSuccess, a \right) \right) \right. \hspace{1.8in} \\ \left. \successSequence \left(
        \sum_{s \in \stateSpace}\sum_{a \in \actionSpace} \frac{\probability_{\coupledTransducer}\left( s, \noSuccess, a \right)}{\sum_{s' \in \stateSpace}\sum_{a' \in \actionSpace} \probability_{\coupledTransducer}\left( s', \noSuccess, a' \right)}\transition{\coupledTransducer}\left( \star, \left( s, \noSuccess, a \right) \right)
      \right)
    \right).
  \end{multline}
\end{definition}
One should think of this as a sequence of probabilities of the form
\begin{equation}
  p(\text{success occurs for the first time on time step $n$}).
\end{equation}

Definition \ref{def.infraBoundedSumSequences} lets us extend the notion of a mixture to bounded sum sequences, suggesting the following lemma.
\begin{lemma}[Success sequences respect mixtures]
  \label{lem.successSequenceRespectsMixtures}
  Let $\coupledTransducer_{k} \colon \coupledTransducers$ be coupled transducers and $\alpha_{k} \geq 0$, such that $\sum_{k = 0}^{\infty} \alpha_{k} = 1$. Success sequences respect mixtures, that is
  \newcommand{\kthSuccessSequence}[1]{\successSequence\left( \coupledTransducer_{#1} \right)}
  \begin{equation}
    \successSequence\left( \mixture{\alpha_}{\coupledTransducer_} \right) = \mixture{\alpha_}{\kthSuccessSequence}.
  \end{equation}

  \begin{proof}
    The first projection of the left hand side (i.e.\ the first element of the pair returned by $\sigma$) is
    \begin{equation}
      \sum_{s \in \stateSpace}\sum_{a \in \actionSpace} \probability_{\mixture{\alpha_}{\coupledTransducer_}}\left( s, \success, a \right) =
      \sum_{s \in \stateSpace}\sum_{a \in \actionSpace} \mixture{\alpha_}{\kthProbabilityAt{\coupledTransducer}{s, \success, a}},
    \end{equation}
    which after simple operations on sums becomes
    \newcommand{\kthTotalSuccessProbability}[1]{\sum_{s \in \stateSpace}\sum_{a \in \actionSpace} \probability_{\coupledTransducer_{#1}}\left( s, \success, a \right)}
    \begin{equation}
      \mixture{\alpha_}{\kthTotalSuccessProbability},
    \end{equation}
    which is exactly the first projection of the right hand side.

    The second projection is the slightly more complex
    \begin{multline}
      \left( \sum_{s \in \stateSpace}\sum_{a \in \actionSpace} \sum_{k = 0}^{\infty} \alpha_{k} \probability_{\coupledTransducer_{k}}\left( s, \noSuccess, a \right) \right) \\
      \successSequence \left(
        \sum_{s \in \stateSpace}\sum_{a \in \actionSpace} \frac{\sum_{k = 0}^{\infty} \alpha_{k} \probability_{\coupledTransducer_{k}}\left( s, \noSuccess, a \right)}{\sum_{s' \in \stateSpace}\sum_{a' \in \actionSpace} \sum_{k = 0}^{\infty} \alpha_{k} \probability_{\coupledTransducer_{k}}\left( s', \noSuccess, a' \right)} \right.\\
      \left. \sum_{k = 0}^{\infty} \frac{\alpha_k \probability_{\coupledTransducer_{k}}\left( s, \noSuccess, a \right)}{\sum_{k' = 0}^{\infty} \alpha_k' \probability_{\coupledTransducer_{k'}}\left( s, \noSuccess, a \right)} \transition{\coupledTransducer_{k}}\left( \star, \left( s, \noSuccess, a \right) \right)
      \right).
    \end{multline}
    Canceling out the numerator of the first fraction with the denominator of the second we get
    \begin{multline}
      \left( \sum_{s \in \stateSpace}\sum_{a \in \actionSpace} \sum_{k = 0}^{\infty} \alpha_{k} \probability_{\coupledTransducer_{k}}\left( s, \noSuccess, a \right) \right) \\
      \successSequence \left(
        \sum_{s \in \stateSpace}\sum_{a \in \actionSpace} \sum_{k = 0}^{\infty} \frac{\alpha_{k} \probability_{\coupledTransducer_{k}}\left( s, \noSuccess, a \right)}{\sum_{s' \in \stateSpace}\sum_{a' \in \actionSpace} \sum_{k' = 0}^{\infty} \alpha_{k'} \probability_{\coupledTransducer_{k'}}\left( s', \noSuccess, a' \right)}
        \transition{\coupledTransducer_{k}}\left( \star \left( s, \noSuccess, a \right) \right)
      \right).
    \end{multline}
    Expanding $1$ into a fraction with ${\sum_{s' \in \stateSpace}\sum_{a' \in \actionSpace} \probability_{\coupledTransducer_{k}}\left( s', \noSuccess, a' \right)}$ for every $k$ inside of the corecursive function and shuffling around some terms yields
    \begin{multline}
      \left( \sum_{s \in \stateSpace}\sum_{a \in \actionSpace} \sum_{k = 0}^{\infty} \alpha_{k} \probability_{\coupledTransducer_{k}}\left( s, \noSuccess, a \right) \right) \\
      \successSequence \left(
        \sum_{k = 0}^{\infty} \frac{\sum_{s' \in \stateSpace}\sum_{a' \in \actionSpace} \alpha_{k} \probability_{\coupledTransducer_{k}}\left( s', \noSuccess, a' \right)}{\sum_{s' \in \stateSpace}\sum_{a' \in \actionSpace} \sum_{k' = 0}^{\infty} \alpha_{k'} \probability_{\coupledTransducer_{k'}}\left( s', \noSuccess, a' \right)} \right.\\
      \left. \sum_{s \in \stateSpace}\sum_{a \in \actionSpace} \frac{\probability_{\coupledTransducer_{k}}\left( s, \noSuccess, a \right)}{\sum_{s' \in \stateSpace}\sum_{a' \in \actionSpace} \probability_{\coupledTransducer_{k}}\left( s', \noSuccess, a' \right)} \transition{\coupledTransducer_{k}}\left( \star, \left( s, \noSuccess, a \right) \right)
      \right),
    \end{multline}
    and finally using the coinductive hypothesis, cancelling the first term with the first denominator, and reorganizing some sums we get
    \begin{multline}
      \sum_{k = 0}^{\infty} \alpha_{k} \left( \sum_{s \in \stateSpace}\sum_{a \in \actionSpace} \probability_{\coupledTransducer_{k}}\left( s, \noSuccess, a \right) \right)\\
      \successSequence \left(
        \sum_{s \in \stateSpace}\sum_{a \in \actionSpace} \frac{\probability_{\coupledTransducer_{k}}\left( s, \noSuccess, a \right)}{\sum_{s' \in \stateSpace}\sum_{a' \in \actionSpace} \probability_{\coupledTransducer_{k}}\left( s', \noSuccess, a' \right)} \transition{\coupledTransducer_{k}}\left( \star, \left( s, \noSuccess, a \right) \right)
      \right),
    \end{multline}
    which is exactly the second term of the right hand side, finishing the proof.
  \end{proof}
\end{lemma}
\newcommand{\bssSum}{\Sigma}
This lets us define the probability of success. We start by defining the sum of a BSS:
\begin{definition}[Sum of a BSS]
  Given a BSS, we define its \emph{$n$-step sum}
  \begin{equation}
      \bssSum_{n} \colon \boundedSumSequences{r_{0}} \to \left[ 0, r_{0} \right]
  \end{equation}
  inductively with the base
  \begin{equation}
    \bssSum_{0}\left( \left( r, t \right) \right) = r
  \end{equation}
  and the inductive step
  \begin{equation}
    \bssSum_{n+1}\left( \left( r, t \right) \right) = r + \bssSum_{n}\left( t \right).
  \end{equation}
  A simple inductive argument shows that this is well defined.

  The sequence $\left( \bssSum_{n}\left( \left( r, t \right) \right) \right)_{n \in \N}$ is obviously bounded and another simple inductive argument shows that it is non-decreasing, letting us define the \emph{sum}
  \begin{equation}
    \bssSum\left( \left( r, t \right) \right) = \lim_{n \to \infty} \bssSum_{n}\left( \left( r, t \right) \right).
  \end{equation}
\end{definition}
This, again, respects mixtures.
\begin{lemma}[Sum of BSSes respect mixtures]
  \label{lem.sumOfMixtures}
  Let $t_k \colon \boundedSumSequences{r_{0, k}}$ be BSSes and $\alpha_{k}
  \geq 0$, such that $\sum_{k = 0}^{\infty} \alpha_{k} = 1$. The sum of
  BSSes respects mixtures, that is
  \newcommand{\kthBssSum}[1]{\bssSum\left( t_{#1}\right)}
  \begin{equation}
    \bssSum\left( \mixture{\alpha_}{t_} \right) =
    \mixture{\alpha_}{\kthBssSum}
  \end{equation}

  \begin{proof}
    It suffices to note that converging limits respect mixtures and perform a simple inductive proof
    \newcommand{\kthProbabilitySuccessSequence}[2]{\bssSum_{#1}\left( \left( r_{#2}, t_{#2} \right) \right)}
    \newcommand{\kthBss}[1]{\left( r_{#1}, t_{#1}\right)}
    \begin{equation}
      \bssSum_{0}\left( \mixture{\alpha_}{\kthBss} \right) = \mixture{\alpha_}{r_} = \mixture{\alpha_}{\kthProbabilitySuccessSequence{0}}
    \end{equation}
    with the step
    \newcommand{\kthProbabilityPlainSuccessSequence}[2]{\bssSum_{#1}\left( t_{#2} \right)}
    \begin{multline}
      \bssSum_{n+1}\left( \mixture{\alpha_}{\kthBss} \right) =
      \mixture{\alpha_}{r_} + \bssSum_{n}\left( \mixture{\alpha_}{t_} \right) =\\
      \mixture{\alpha_}{r_} + \mixture{\alpha_}{\kthProbabilityPlainSuccessSequence{n}} =
      \mixture{\alpha_}{\kthProbabilitySuccessSequence{n+1}}.
    \end{multline}
  \end{proof}
\end{lemma}
With this, we can define the success probability:
\begin{definition}
  \label{def.success}
  We define the \emph{success probability} of a policy-environment pair as the composition
  \begin{equation}
    \successProbability\left( \policy, \environment \right) = \bssSum\left( \successSequence\left( \coupling\left( \policy, \environment \right) \right) \right)
  \end{equation}
  and analogously for $n$-step success probabilities.
\end{definition}
Because of the form of the inductive definition success probability satisfies the following equation:
\begin{equation}
  \label{eq.success}
  \successProbability\left( \policy, \environment \right) =
  \sum_{s \in \stateSpace} \sum_{a \in \actionSpace}
  \probability_{\coupling\left( \policy, \environment \right)}\left( s, \success, a \right) +
  \probability_{\coupling\left( \policy, \environment \right)}\left( s, \noSuccess, a \right) \successProbability\left( \transition_{\coupling\left( \policy, \environment \right)}\left( \star, \left( s, \noSuccess, a \right) \right) \right).
\end{equation}
Intuitively the first term in the definition of probability of success represents the probability of succeeding in a single step, while the second term corresponds to succeeding at any point in the future, assuming no success has been achieved in this step.

One could be tempted to use equation~\eqref{eq.success} as a `corecursive' definition of success probability, but that would not work, because the codomain $\left[ 0, 1 \right]$ is not coinductively defined by $\successProbability$. Indeed, this equation is also satisfied by the constant function $1$.

To make the meaning of equation~\eqref{eq.success} clearer consider the following equations for the $n$-step probability of success of a specific policy-environment pair.
\begin{align*}
  \successProbability_0\left( \policy, \environment \right) &= \sum_{s \in \stateSpace} \sum_{a \in \actionSpace} \probability_{\coupling\left( \policy, \environment \right)}\left( s, \success, a \right),\\
  \successProbability_{n+1}\left( \policy, \environment \right) &= \sum_{s \in \stateSpace} \sum_{a \in \actionSpace}
                                                                  \probability_{\coupling\left( \policy, \environment \right)}\left( s, \success, a \right) +
                                                                  \probability_{\coupling\left( \policy, \environment \right)}\left( s, \noSuccess, a \right) \successProbability_{n}\left( \transition_{\coupling\left( \policy, \environment \right)}\left( \star, \left( s, \noSuccess, a \right) \right) \right).
\end{align*}
\newcommand{\failureProbability}{F}
\newcommand{\stepFailedEvolutionDefinition}[1]{\environment \evolution \interleaved{ \elementString{a}_{\leq #1}}{ \interleaved{ \elementString{s}_{\leq #1}}{ \noSuccess^{#1} } }}
\newcommand{\stepFailedEvolution}[1]{as(#1)}
For a given evolution $\interleaved{ \elementString{s}}{ \elementString{a} } \in \left(\stateSpace \times \actionSpace\right)^{n}$ we can define the failure probability along this evolution
\begin{equation}
  \failureProbability_{\coupling\left( \policy, \environment \right)}\left( \elementString{s}, \elementString{a} \right) = \begin{cases}
    \prod_{i = 0}^n \probability_{\stepFailedEvolution{i}}\left( \elementString{s}_{i + 1}, \noSuccess \right) & \textrm{ if } \policy \evolution \interleaved{ \elementString{s}}{ \elementString{a} } \textrm{ is valid,}\\
    0 & \textrm{ otherwise.}
  \end{cases}
\end{equation}
The subscript $\stepFailedEvolution{i}$ is defined as $\stepFailedEvolutionDefinition{i}$ and represents the evolution of the environment following $\interleaved{ \elementString{s}}{ \elementString{a} }$ for $i$ steps and assuming no success at every step.

Note that this function is zero if the environment evolution is not valid, because one of the elements of the product will be zero.

Now, by expanding equation \eqref{eq.success} $n$ times, we can write down an alternative formulation of the success probability for each $n$
\begin{equation}
  \label{eq.successBellman}
  \successProbability\left( \policy, \environment \right) = \successProbability_{n}\left( \policy, \environment \right) + \sum_{\elementString{s} \in \stateSpace^{n}} \sum_{\elementString{a} \in \actionSpace^{n}} \failureProbability_{\coupling\left( \policy, \environment \right)}\left( \elementString{s}, \elementString{a} \right) \successProbability\left( \policy \evolution \interleaved{ \elementString{s}}{ \elementString{a} }, \environment \evolution \interleaved{ \elementString{a}}{ \interleaved{ \elementString{s}}{ \noSuccess^{n} } } \right),
\end{equation}
which is a sum of the probability of success within the first $n$ steps, plus the expected value of the probability of success after a valid $n$-step evolution which did not achieve success.

Probability of success also has the very useful property of respecting mixtures both in policies as well as environments.
\newcommand{\politicalSuccessProbability}[1]{\successProbability\left( \policy_{#1}, \environment \right)}
\newcommand{\kthSuccessSequence}[1]{\left( r_{#1}, t_{#1} \right)}
\begin{lemma}[Success probabilities respect mixtures]
  \label{lem.successOfMixtures}
  Let $\policy_{k} \colon \policies$ be policies, $\environment_{k} \colon \environments$ be environments, and $\alpha_{k}, \beta_{k} \geq 0$ such that $\sum_{k = 0}^{\infty}\alpha_{k} = 1$ and $\sum_{k = 0}^{\infty} \beta_{k} = 1$. The success probability function respects mixtures, that is
  \begin{equation}
    \successProbability\left( \mixture{\alpha_}{\policy_}, \mixture{\beta_}{\environment_} \right) =
    \sum_{k = 1}^{\infty} \sum_{l = 1}^{\infty} \alpha_{k}\beta_{l} \successProbability\left( \transducer_{k}, \environment_{l} \right)
  \end{equation}

  \begin{proof}
    By Lemma \ref{lem.couplingOfMixtures} coupling respects mixtures, by
    Lemma \ref{lem.successSequenceRespectsMixtures} success sequences
    respect mixtures, and by Lemma \ref{lem.sumOfMixtures} sums of BSSes
    respect mixtures, thus so does their composition.
  \end{proof}
\end{lemma}

With success defined and explored we can now define what it means for a policy to be optimal for an environment.
\begin{definition}[Optimality]
  \label{def.teleoOptimality}
  Let $\constrainedTransducers \subseteq \policies$ be a constrained set of policies. We will say that a policy $\policy \colon \constrainedTransducers$ is \emph{$\constrainedTransducers$-optimal} for a teleo-environment $\environment \colon \environments$ iff
  \begin{equation}
    \forall_{\policy' \colon \constrainedTransducers} \successProbability\left( \policy', \environment \right) \leq \successProbability\left( \policy, \environment \right).
  \end{equation}
  In other words $\policy$ maximizes the success probability in $\environment$
  among the policies in $\constrainedTransducers$. We will sometimes
  omit the `$\,\constrainedTransducers$-', when it is obvious from context.
\end{definition}

\section{Filtering}
\label{sec.filtering}

We remarked in Section \ref{sec.evolutionIsFiltering} that the evolution operator can be seen in terms of Bayesian filtering. We will now explain this intuition in more detail, and derive some formal results.

An optimal policy $\policy$ for $\environment$ represents the behavioural propensities of an ideally-performing agent with beliefs and values encoded by $\environment$. In this sense, it is consistent with $\policy$ to attribute a `mental state' represented by $\environment$ (under the assumption that $\policy$ behaves `rationally'). We can say, as it were, that $\environment$ is a `permissible' mental state attribution for $\policy$, at least for practical purposes.

Iterated application of the evolution operator for a policy $\policy$ produces a sequence of transducers $\policy_1, \policy_2, \cdots$, and so on, where $\policy_{n+1} = \policy_n \evolution (s_n, a_n)$ for some $(s_n, a_n)$. This section will establish that, if it is permissible to attribute a `mental state' $\environment_1$ to $\policy_1$, then it is permissible to attribute some $\environment_n$ to each $\policy_n$ so that the sequence $\environment_1, \environment_2, \cdots$ is what would be obtained by a process resembling Bayesian filtering, beginning with a prior corresponding to $\environment_1$, and proceeding at each step by updating on an observation corresponding to $(a_n, (s_n, \bot))$.

It is worth pointing out that this differs from standard Bayesian filtering in that it is `value-laden': according to the sequence $\environment_1, \environment_2, \cdots$, the transition $\policy_n \to \policy_{n+1}$ `behaves as though' an agent were updating on a notional observation $\bot$ over the environment's `telos' channel, in addition to observations of $\policy_n$'s actual input $s_n$ and output $a_n$. Recall that $\policy_n$ does not take a telos signal as an input, although $\environment_n$ produces it as an output. We will see in Section \ref{sec.pureFiltering} that if there exists any $\environment$ for which $\policy_1$ is optimal, it is also admissible to attribute a sequence $\environment_1, \environment_2, \cdots$ according to which $\policy_n \to \policy_{n+1}$ `behaves as though' the agent were updating only on the actual values $s_n$ and $a_n$ (and not an imaginary observation~$\bot$).

The reason for the value-laden nature of this filtering can be explained as follows. The way to achieve the best chances of future success in a new environment may depend on whether or not the current environment emitted the success signal. But the agent's goal is to achieve success at least once, with success signals after that counting for nothing. If success has been attained in the current time step, there is consequently no advantage to pursuing an optimal policy in the next step -- if the agent doesn't know whether its goal was attained, it shouldn't be distracted by the possibility that it was. Hence, the agent should always behave as though it has not already achieved success, in order to maximise its overall chances of eventual success.
We will discuss this more in Section \ref{sec.noGenericSensorimotor}.

Our formal treatment considers (value-laden) filtering for the more general case of constrained policies (i.e.\ ones that are $T$-optimal for some set $T$) rather than only policies that are globally optimal. We will begin by defining something we call the `value-laden Bellman property'.

\begin{definition}[Value-laden Bellman property for teleo-optimality]
  \label{def.filtering}
  \label{def.valueLadenBellman}
  Let $\constrainedTransducers \subseteq \policies$ be a set of constrained
  policies. We say that $\constrainedTransducers$ has the \emph{value-laden Bellman property
    for teleo-optimality} (or simply the \emph{Bellman property} when there is no ambiguity) if for any policy $\policy \colon \constrainedTransducers$
  that is $\constrainedTransducers$-optimal for an environment $\environment
  \colon \environments$ and any trajectory $\left( \elementString{a},
    \elementString{s} \right) \in \actionSpace^{n} \times \stateSpace^{n}$ of any
  length $n$, if $\environment' = \environment \evolution \interleaved{
    \elementString{a}}{ \interleaved{ \elementString{s}}{ \noSuccess^{n} } }$
  and $\policy' = \policy \evolution \interleaved{ \elementString{s}}{
    \elementString{a} }$ are valid  then $\policy'$ is $\constrainedTransducers'$-optimal for $\environment'$, where $\constrainedTransducers' = \constrainedTransducers
  \evolution \interleaved{ \elementString{s}}{\elementString{a} }$.
\end{definition}

This property says that if an agent is optimal within its class for one problem (in the sense of maximising success probability), then after it has interacted with the environment for some time it should still be optimal in its new, resulting environment.
This is essentially Bellman's principle of optimality \cite[Chapter III, \S 3]{Bellman57}.
However, a difference in our case is that to obtain the environment for the next time step we condition not only on the agent's action and received sensor value, but also on success not occurring in the current time step.
As noted above, this is because the agent's goal is to achieve success at least once.
The agent does not know whether success has been achieved or not, but its future actions will only matter in the case where success has not already been achieved.
We will say more on this point below.

We want to know which sets of constrained policies have the Bellman property. We can prove the following theorem.
\begin{theorem}[Value-laden Bellman theorem for teleo-optimality]
  \label{thm.filtering}
  Let $\constrainedTransducers \subseteq \policies$ be a set of constrained
  policies. If $T$ is closed under trajectory splicing (see Definition \ref{def.cuts}), it has the Bellman property.
\end{theorem}
\begin{corollary}
  \label{cor.allTransducersHasBellmanProperty}
  The set of all transducers $\policies$ has the Bellman property.
\end{corollary}

\begin{proof}[Proof of Theorem~\ref{thm.filtering}]
First recall that for any trajectory $\interleaved{ \elementString{s} }
  { \elementString{a} } \in \left( \stateSpace \times \actionSpace \right)^{n}$ if
  $\constrainedTransducers$ is closed under trajectory splicing then so is
  $\constrainedTransducers \evolution
  \interleaved{\elementString{s}}{\elementString{a}}$, by induction on the
  length of the trajectory, as mentioned after Definition
  \ref{def.cutsCounductive}.

    Let us say that $\constrainedTransducers \subseteq \policies$ has the $n^\text{th}$ order Bellman property if definition~\ref{def.valueLadenBellman} holds for a specific value of $n$, rather than all $n$.

    We will prove theorem~\ref{thm.filtering} by showing that the $n^\text{th}$ order Bellman property holds for all $n$, which we do by induction on $n$.

    For $n = 0$ the property trivially holds.
    Let us then assume that $T$ has the $(n - 1)^\text{th}$ order Bellman property.
  Fix a a policy $\pi$ that is $T$-optimal for a teleo-environment $\varepsilon$, as well as a trajectory $\interleaved{ \elementString{a} }{ \elementString{s} } \in \left( \actionSpace \times \stateSpace \right)^{n}$ of length $n$ that is valid for evolving them. We want to show that $\policy' = \policy \evolution \interleaved{\elementString{s}}{\elementString{a}}$ is $\constrainedTransducers'$-optimal for $\environment' = \evolution \interleaved{ \elementString{a}}{ \interleaved{
      \elementString{s}}{ \noSuccess^{n} }}$, where $\constrainedTransducers' = \constrainedTransducers \evolution \interleaved{\elementString{s}}{\elementString{a}}$.

  Let us write $\policy'' = \policy \evolution \interleaved{
    \elementString{s}_{<n}}{ \elementString{a}_{<n} }$, and $\environment'' =
  \environment \evolution \interleaved{ \elementString{a}_{<n}}{ \interleaved{
      \elementString{s}_{<n}}{ \noSuccess^{n-1} } }$. Then, by the inductive hypothesis,
  $\policy''$ is $\constrainedTransducers''$-optimal for
  $\environment''$, where $\constrainedTransducers'' = \constrainedTransducers \evolution
    \left(\elementString{s}_{<n}, \elementString{a}_{<n} \right)$. Since $\policy' = \policy'' \evolution \left(
    \elementString{s}_{n}, \elementString{a}_n \right)$, $\environment' =
  \environment'' \evolution \left( \elementString{a}_n, \left(
      \elementString{s}_n, \noSuccess \right) \right)$, and $\constrainedTransducers''$ is closed
    under trajectory splicing,
    all we need to show is that $\constrainedTransducers''$ has the $1^\text{st}$ order Bellman property.
    In order to show this, we show that if any $\constrainedTransducers \subseteq \policies$ is closed under trajectory splicing, then it has the $1^\text{st}$ order Bellman property. The full Bellman property will then follow.

  To show this, fix some arbitrary constrained set $\constrainedTransducers$ of
  transducers that is closed under trajectory splicing, together with a new
  $\policy \colon \constrainedTransducers$ that is $\constrainedTransducers$-optimal for some $\environment$, as well as some pair $(s,a) \in \stateSpace \times \actionSpace$ valid for evolving them. Assume that $\policy' = \policy \evolution \left( s, a \right) =
  \transition_{\policy}\left( s, a \right)$ is not $\constrainedTransducers \evolution \left( s, a \right)$-optimal for $\environment' =
  \environment \evolution \left( a, \left( s, \noSuccess \right) \right)$,
  i.e.~there exists a policy $\policy^{*'} \colon \constrainedTransducers \evolution \left( s, a \right)$, such that $\successProbability\left( \policy^{*'}, \environment' \right) > \successProbability\left( \transition_{\policy}\left( s, a \right), \environment' \right)$. We will construct a policy $\policy^{*} \colon \constrainedTransducers$ such that $\successProbability\left( \policy^{*}, \environment \right) > \successProbability\left( \policy, \environment \right)$, thus contradicting the fact that $\policy$ is optimal for $\environment$.

  Define this policy as $\policy^{*} = \policy \splicedAlong{\left( s, a
    \right)} \policy^{*'}$. This belongs to $\constrainedTransducers$ since
  $\constrainedTransducers$ is closed under trajectory splicing.

  Now it remains to show that $
  \successProbability\left( \policy^{*}, \environment \right) < \successProbability\left( \policy, \environment \right).
  $
  Using equation \eqref{eq.success} considering the form of $\policy^{*}$ we get
  \begin{equation}
    \successProbability\left( \policy^{*}, \environment \right) =
    \sum_{u \in \stateSpace} \sum_{b \in \actionSpace}
    \probability_{\coupling\left( \policy, \environment \right)}\left( u, \success, b \right) +
    \probability_{\coupling\left( \policy, \environment \right)}\left( u, \noSuccess, b \right) \successProbability\left( \transition_{\policy^{*}}\left( u, b \right), \transition_{\environment}\left( b, \left( u, \noSuccess \right) \right) \right).
  \end{equation}
  Expanding the definition of $\transition_{\policy^{*}}$ and separating the term for $\left( a, s \right)$ we get
  \begin{multline}
    \successProbability\left( \policy^{*}, \environment \right) =\\
    \sum_{u \in \stateSpace} \sum_{\substack{b \in \actionSpace\\ \left( b, u \right) \neq \left( a, s \right)}} \left(
      \probability_{\coupling\left( \policy, \environment \right)}\left( u, \success, b \right) +
      \probability_{\coupling\left( \policy, \environment \right)}\left( u, \noSuccess, b \right) \successProbability\left( \transition_{\policy}\left( u, b \right), \transition_{\environment}\left( b, \left( u, \noSuccess \right) \right) \right)
    \right) \\
    + \probability_{\coupling\left( \policy, \environment \right)}\left( s, \success, a \right) +
    \probability_{\coupling\left( \policy, \environment \right)}\left( s, \noSuccess, a \right) \successProbability\left( \policy^{*'}, \transition_{\environment}\left( a, \left( s, \noSuccess \right) \right) \right).
  \end{multline}
  Noting that $\transition_{\environment}\left( a, \left( s, \noSuccess \right) \right) = \environment'$ and using the inequality on probabilities of success considering that $\probability_{\coupling\left( \policy, \environment \right)}\left( s, \noSuccess, a \right) \neq 0$ because the evolution of the environment was valid, we can now write
  \begin{multline}
    \successProbability\left( \policy^{*}, \environment \right) >\\
    \sum_{u \in \stateSpace} \sum_{\substack{b \in \actionSpace\\ \left( b, u \right) \neq \left( a, s \right)}} \left(
      \probability_{\coupling\left( \policy, \environment \right)}\left( u, \success, b \right) +
      \probability_{\coupling\left( \policy, \environment \right)}\left( u, \noSuccess, b \right) \successProbability\left( \transition_{\policy}\left( u, b \right), \transition_{\environment}\left( b, \left( u, \noSuccess \right) \right) \right)
    \right) \\
    + \probability_{\coupling\left( \policy, \environment \right)}\left( s, \success, a \right) +
    \probability_{\coupling\left( \policy, \environment \right)}\left( s, \noSuccess, a \right) \successProbability\left( \transition_{\policy}\left( s, a \right), \transition_{\environment}\left( a, \left( s, \noSuccess \right) \right) \right),
  \end{multline}
  and after reintegrating the last term we can use equation \eqref{eq.success} to get $
    \successProbability\left( \policy^{*}, \environment \right) > \successProbability\left( \policy, \environment \right),
  $
  finishing the proof.
\end{proof}

While being closed under trajectory splicing is a sufficient condition for the
Bellman property to hold it is not necessary. Consider the one-flip transducers
defined in Example \ref{ex.oneFlip}. They are not closed under trajectory
splicing, but we assert (without a formal proof) that the Bellman theorem still holds for them. This is due to the fact
that, as we will see later, for any environment there is always an optimal
deterministic policy, together with the fact that
one-flip transducers are closed under a `limited version'
of trajectory splicing, where you only replace future policies with
deterministic ones.

The fact that transducers obey the (value-laden) Bellman theorem has an important consequence.
Let us take stock of what it entails.

Our thesis is that if a physical system is optimal for some problem (specified, in our case, as a teleo-environment), then we can attribute beliefs and goals to the system that correspond to the given problem.
In our case, this means that if a policy $\pi$ is optimal for a given teleo-environment $\varepsilon$, then it is permissible to attribute to $\pi$ the goal of achieving the telos signal at least once, and the belief that the dynamics of the environment (and the telos signal) are given by $\varepsilon$.
Corollary~\ref{cor.allTransducersHasBellmanProperty} says that attributing beliefs in this way is consistent with Bayesian updating in the following sense:

Suppose we can attribute beliefs $\varepsilon$ to $\pi$, and that $\pi$ then emits action $a:A$ and receives sensor value $s:S$ from the environment, evolving into $\pi\bullet(s,a)$.
It is then permissible to attribute beliefs $\varepsilon\bullet(a,s,\bot)$ to $\pi\bullet(s,a)$.
We have established in Section \ref{sec.transducers} that evolving a transducer is closely related to Bayesian filtering.
Thus, after the update, $\pi\bullet(s,a)$ can be attributed beliefs that corresponding to performing a Bayesian filtering update on the environment.

In this story the filtering step conditions not only on the agent's received sensor data $s$ but also on success not being achieved on the current time step (i.e.\ the $\bot$ symbol), even though the agent doesn't have access to the telos channel, and indeed success might actually have occurred.
We argued that we can understand this `from the agent's point of view' by noting that since the agent's (attributed) goal is to achieve success at least once, if success \emph{has} already occurred, the agent's actions no longer matter.
As we will eventually see, in Example~\ref{ex.noSensorimotorOnlyFiltering}, if an agent would condition only on $s$ and $a$ but not $\bot$, it would not necessarily be optimal for the resulting conditioned environment.

\subsection{UFS transducers do not filter}
\label{sec.noUfsFiltering}

There are classes of transducers for which the Bellman property does not hold. For a fairly general and nontrivial example we will look at unifilar finite state transducers.

The somewhat simpler absent-minded driver example in Section \ref{sec.absentMinded} could also
be used to illustrate a situation in which the Bellman
property does not hold. However, its simplicity might give the incorrect
impression that this lack is a corner case, both in terms of the class being
single state transducers, as well as the optimal policy being nondeterministic,
which is completely unrelated to filtering.

First we will require some definitions, which will also be used for future examples and proofs.
\begin{definition}[Success and nothing distributions]
  \label{def.successDistribution}
  Let $\stateSpace$ be an $n$-element state space. We define the \emph{(uniform) success and nothing distributions} on $\stateSpace \times \telosSpace$ as
  \begin{equation}
    \successDistribution\left( s, t \right) = \begin{cases}
      \frac{1}{n} & \textrm{ if } t = \success,\\
      0 & \textrm{ otherwise.}
    \end{cases}
  \end{equation}
  and
  \begin{equation}
    \nothingDistribution\left( s, t \right) = \begin{cases}
      \frac{1}{n} & \textrm{ if } t = \noSuccess,\\
      0 & \textrm{ otherwise.}
    \end{cases}
  \end{equation}
  respectively.
\end{definition}
\newcommand{\despairEnvironment}{\environment_{\frac{1}{2}}}
\begin{definition}[Doom and despair environments]
  \label{def.doomAndDespair}
  The \emph{doom environment} is one that never returns success, formally defined as
  \begin{equation}
    \doomEnvironment = \left( \nothingDistribution, \left( a, \left( s, t \right) \right) \mapsto \doomEnvironment \right),
  \end{equation}
  while the \emph{despair environment} has $\frac{1}{2}$ probability of success for one step and then changes into the doom environment, regardless of which action is taken
  \begin{equation}
    \despairEnvironment = \left( \frac{1}{2}\nothingDistribution + \frac{1}{2}\successDistribution, \left( a, \left( s, t \right) \right) \mapsto \doomEnvironment \right).
  \end{equation}
\end{definition}
\begin{example}[No UFS filtering]
  \label{ex.noUfsFiltering}
  Let $\stateSpace = \left\{ 1, \dots, n \right\}$, $\actionSpace = \left\{ 1, \dots, n, n+1 \right\}$ and consider only policies constrained to $\unifilarFiniteStateTransducers{n} \subseteq \policies$. We will build an explicit counterexample to filtering in this setting, defining the necessary environment step by step.

  \newcommand{\mimicEnvironment}{\environment_m}
  First we define a family of (weighted) mimic environments indexed by states
  \begin{equation}
    \mimicEnvironment\left( s' \right) = \left( \frac{2^{s'} - 1}{2^{s'}}\nothingDistribution + \frac{1}{2^{s'}}\successDistribution, \left( a, \left( s, t \right) \right) \mapsto \begin{cases}
      \mimicEnvironment\left( s \right) & \textrm{ if } a = s',\\
      \despairEnvironment & \textrm{ if } a = n + 1,\\
      \doomEnvironment & \textrm{ otherwise}
    \end{cases} \right).
\end{equation}
The name refers to the fact that an optimal policy will have to mimic what the environment is doing:
if the previous sensor input was $k$ then the next action should also be $k$.

It now remains to define the environment itself
\begin{equation}
  \environment = \left( \nothingDistribution, \left( a, \left( s, t \right) \right) \mapsto \begin{cases}
    \mimicEnvironment\left( s \right) & \textrm{ if } a = n + 1,\\
    \doomEnvironment & \textrm{ otherwise}
  \end{cases} \right).
\end{equation}
In other words, the environment never returns success on the first time step, and the action must be $n+1$.

The policy induced by the unifilar $n$ state machine (with $\unifilarStateSet = \actionSpace$)
\begin{equation}
  \policy^{*'} = \left( s', x \mapsto x, \left( x, s, a \right) \mapsto s \right)
\end{equation}
is optimal for every mimic environment $\mimicEnvironment\left( s' \right)$. (We abuse notation and write $s'$ for the deterministic distribution always returning $s'$.) This policy has success probability $1$, since it can succeed with a nonzero constant expected probability in every step, so it will eventually succeed. Any other policy will have a nonzero probability of not repeating the previous state as its action, which would make it fail forever (since then the environment gives at most one constant probability of success and afterwards behaves like the doom environment), giving a probability of success below $1$.

The policy induced by the unifilar $n$ state machine
\begin{equation}
  \policy = \left( n, x \mapsto \begin{cases}
    x & \textrm{ if } x \neq n,\\
    n+1 & \textrm{ if } x = n
  \end{cases}, \left( x, s, a \right) \mapsto s \right).
\end{equation}
is optimal for the environment itself.

This transducer is (uniquely) optimal for $\environment$ within the class of $n$-state UFS transducers but it is not optimal among all transducers.
(In fact there is an $(n+1)$-state UFS transducer that outperforms it.)

We argue that $\policy$ is optimal informally rather than giving a proof. The policy has to return $n+1$ as its first action, otherwise its probability of success is $0$. This constrains the output function to map some memory state to $n+1$ making it impossible for the future function to be a perfect mimic. Mimicking the first $n-1$ states/actions is the uniquely optimal solution, because they have the best chance of providing success when they occur, while all states have the same chance of occurring. At the same time when we no longer can mimic the environment, the policy is optimal when it returns $n + 1$ for another shot at a success. In the end the success probability of such an imperfect mimic is $\frac{2^{n-1} + 2^{n-2} - 1}{2^{n} - 1}$ and any other unifilar $n$ state will have a lower result, although it is tedious to check.

To show that filtering is violated, we need to find some policy that is the result of evolving a
$\unifilarFiniteStateTransducers{n}$ policy by $\left( s, n + 1 \right)$, such that it has a greater success probability than  $\policy \evolution \left( s, n + 1 \right)$.
  This policy cannot be $\policy^{*'}$, since it can't be obtained by evolving a
$\unifilarFiniteStateTransducers{n}$ policy by $\left( s, n + 1 \right)$.
This follows from the fact that $\policy^{*'}\in \unifilarFiniteStateTransducers{n}\setminus\unifilarFiniteStateTransducers{n-1}$ and the fact that $\policy^{*'}$ has no state in which it has a nonzero
  probability of outputting $n+1$.
  However, we can consider a policy $\policy^{*'}_\alpha$ that acts exactly like
  $\policy^{*'}$, except it has probability $\alpha$ of returning action $n + 1$
  in state $n$. If we now make $\alpha$ arbitrarily small the success chance of
  such a policy after the initial (extremely improbable) evolution can be arbitrarily close to $1$.
\end{example}

Interestingly the environment in the above example can itself be defined as a UFS machine with $n + 3$ states -- one for the initial state, one for despair, one for doom, and $n$ for the usual states remembering history. We do not know whether the constant $3$ can be lowered.

\subsection{Sensorimotor-only Filtering}
\label{sec.pureFiltering}

Theorem \ref{thm.filtering} requires that the environment is evolved under the assumption that no success is returned at every step. We will explore how important this assumption really is, and in which situations it can be dropped. First let us define what dropping this assumption would even look like.
\begin{definition}[Success-ambivalent evolution]
  \label{def.genericEvolution}
  Let $\environment \colon \environments$ be an environment. We define its \emph{success-ambivalent evolution} to be
  \begin{multline}
    \environment \evolution \left( a, s \right) =
    \frac{\probability_{\environment}\left( \left( s, \noSuccess \right) \right)}{\probability_{\environment}\left( \left( s, \noSuccess \right) \right) + \probability_{\environment}\left( \left( s, \success \right) \right)} \transition_{\environment}\left( a, \left( s, \noSuccess \right) \right) +\\
    \frac{\probability_{\environment}\left( \left( s, \success \right) \right)}{\probability_{\environment}\left( \left( s, \noSuccess \right) \right) + \probability_{\environment}\left( \left( s, \success \right) \right)} \transition_{\environment}\left( a, \left( s, \success \right) \right),
  \end{multline}
  that is the mixture of transducers resulting from either a success or no success transition proportional to the relative probabilities of such transitions.

  We use the same notation as for normal evolutions, but context should disambiguate sufficiently well to avoid confusing the reader.
\end{definition}
Now consider a special class of environments that achieve success at most once at a given trajectory.
\newcommand{\singleSuccess}{W}
\begin{definition}[Single success environments]
  \label{def.singleSuccess}
  Let $\environment \colon \environments$ be an environment. We call it a \emph{single success environment} if no trajectory $\interleaved{ \elementString{a}}{\interleaved{\elementString{s}}{\elementString{t}}}$ such that there exist two indices $i, j \in \N$ for which $\elementString{t}_i = \elementString{t}_j = \success$ is a valid evolution for $\environment$. We denote the set of all such environments as $\singleSuccess$.
\end{definition}
We can inject any environment into $\singleSuccess$.
\newcommand{\successTruncation}{Z}
\newcommand{\dooming}{D}
\newcommand{\desuccessing}{d}
\begin{definition}[Single success truncation]
  \label{def.successTruncation}
  We define the \emph{single success truncation} to be a function $\successTruncation \colon \environments \to \singleSuccess$ defined as
  \begin{equation}
    \successTruncation\left( \environment \right) = \left( \probability_{\environment}, \left( a, \left( s, t \right) \right) \mapsto \begin{cases}
      \successTruncation\left( \transition_{\environment}\left( a, \left( s, t \right) \right) \right) & \textrm{ if } t = \noSuccess,\\
      \dooming\left( \transition_{\environment}\left( a, \left( s, t \right) \right) \right) & \textrm{ if } t = \success
    \end{cases}\right),
\end{equation}
where $\dooming \colon \environments \to \environments$ is the \emph{dooming} function that removes all successes, defined as
\begin{equation}
  \dooming\left( \environment \right) = \left( \desuccessing\left( \probability_{\environment} \right), \left( a, \left( s, t \right) \right) \mapsto \dooming\left( \environment \evolution \left( a, s \right) \right) \right),
\end{equation}
with $\desuccessing$ defined simply as
\begin{equation}
  \desuccessing\left( p \right)\left( \left( s, t \right) \right) = \begin{cases}
    p\left( \left( s, \success \right) \right) + p\left( \left( s, \noSuccess \right) \right) & \textrm{ if } t = \noSuccess,\\
    0 & \textrm{ otherwise.}
  \end{cases}
\end{equation}
It's easy to see that $\successTruncation$ is an identity on $\singleSuccess$, justifying the name.
\end{definition}
It turns out that, because we only care about success ever occurring, single success truncation preserves success probability.
\begin{lemma}[Truncation preserves success probability]
  \label{lem.successRespectsTruncation}
  Let $\policy \colon \policies$ and $\environment \colon \environments$ be an arbitrary policy-environment pair. Then
  \begin{equation}
    \successProbability\left( \policy, \environment \right) = \successProbability\left( \policy, \successTruncation\left( \environment \right) \right).
  \end{equation}

  \begin{proof}
    We will prove this for success sequences and apply Lemma \ref{lem.successRespectsPoset}.

    Since truncation preserves associated probabilities, it sufficess to inspect the transition function
    \begin{equation}
      \sum_{s \in \stateSpace} \sum_{a \in \actionSpace}
      \probability_{\coupling\left( \policy, \successTruncation\left( \environment \right) \right)}\left( s, \noSuccess, a \right) \successSequence\left( \coupling\left( \transition_{\policy}\left( s, a \right), \transition_{\successTruncation\left( \environment \right)}\left( a, \left( s, \noSuccess \right) \right) \right) \right) =
    \end{equation}
    applying the definition of single success truncation
    \begin{equation}
      \sum_{s \in \stateSpace} \sum_{a \in \actionSpace}
      \probability_{\coupling\left( \policy, \environment \right)}\left( s, \noSuccess, a \right) \successSequence\left( \coupling\left( \transition_{\policy}\left( s, a \right), \successTruncation\left( \transition_{\environment}\left( a, \left( s, \noSuccess \right) \right) \right) \right) \right).
    \end{equation}
    It remains to use the coinductive hypothesis to finish the proof.
  \end{proof}
\end{lemma}
In particular this means that if a policy is optimal for an environment $\environment$, then it is also optimal for its truncated version $\successTruncation\left( \environment \right)$.

There is one more lemma that will prove useful in the upcoming theorem proof.
\begin{lemma}
  \label{lem.evolvingSingleSuccessAssumesFailure}
  Let $\policy \colon \policies$ be a policy, $\environment \colon \singleSuccess$ be a single success environment, and $\left( a, s \right) \in \actionSpace \times \stateSpace$ be an action-state pair. Then
  \begin{equation}
    \successProbability\left( \policy, \environment \evolution \left( a, s \right) \right) = \frac{\probability_{\environment}\left( \left( s, \noSuccess \right) \right)}{\probability_{\environment}\left( \left( s, \noSuccess \right) \right) + \probability_{\environment}\left( \left( s, \success \right) \right)} \successProbability\left( \policy, \environment \evolution \left( a, \left( s, \noSuccess \right) \right) \right).
  \end{equation}

  \begin{proof}
    By the definition of success-ambivalent evolution and Lemma \ref{lem.successOfMixtures} we can write the left hand side as
    \begin{multline}
      \frac{\probability_{\environment}\left( \left( s, \noSuccess \right) \right)}{\probability_{\environment}\left( \left( s, \noSuccess \right) \right) + \probability_{\environment}\left( \left( s, \success \right) \right)} \successProbability\left(\policy, \environment \evolution \left( a, \left( s, \noSuccess \right) \right) \right) +\\
      \frac{\probability_{\environment}\left( \left( s, \success \right) \right)}{\probability_{\environment}\left( \left( s, \noSuccess \right) \right) + \probability_{\environment}\left( \left( s, \success \right) \right)} \successProbability\left( \environment \evolution \left( a, \left( s, \success \right) \right) \right),
    \end{multline}
    but since $\environment$ is single success, $\successProbability\left( \environment \evolution \left( a, \left( s, \success \right) \right) \right) = 0$, leaving
    \begin{equation}
      \frac{\probability_{\environment}\left( \left( s, \noSuccess \right) \right)}{\probability_{\environment}\left( \left( s, \noSuccess \right) \right) + \probability_{\environment}\left( \left( s, \success \right) \right)} \successProbability\left(\policy, \environment \evolution \left( a, \left( s, \noSuccess \right) \right) \right),
    \end{equation}
    as required.
  \end{proof}
\end{lemma}

This allows us to prove the sensorimotor-only filtering result.
\begin{theorem}[Sensorimotor-only Bellman theorem for teleo-optimality]
  \label{thm.pureFiltering}
  Let $\constrainedTransducers \subseteq \policies$ be a set of constrained
  policies that is closed under trajectory splicing. Let $\policy \colon
  \constrainedTransducers$ be a policy $\constrainedTransducers$-optimal for the
  single success environment $\environment \colon \singleSuccess$ and $\left(
    \elementString{a}, \elementString{s} \right) \in \actionSpace^{n} \times
  \stateSpace^{n}$ be a trajectory of length $n$. Then if $\environment' =
  \environment \evolution \interleaved{ \elementString{a}}{ \elementString{s} }$
  and $\policy' = \policy \evolution \interleaved{ \elementString{s}}{
    \elementString{a} }$ are valid, then $\policy'$ is
  $\constrainedTransducers\evolution \interleaved{ \elementString{s}}{
    \elementString{a} }$-optimal for $\environment'$.

  \begin{proof}
    As previously it suffices to prove the claim for evolution by a single pair $\left( a, s \right)$ and get the general statement by induction.

    Given a policy $\policy_{0} \colon \constrainedTransducers$ by the value-laden Bellman theorem we know that
    \begin{equation}
      \successProbability\left( \policy_{0}, \environment \evolution \left( a, \left( s, \noSuccess \right) \right) \right) \leq \successProbability\left( \policy \evolution \left( s, a \right), \environment \evolution \left( a, \left( s, \noSuccess \right) \right) \right).
    \end{equation}
    We can multiply this inequality
    \begin{multline}
      \frac{\probability_{\environment}\left( \left( s, \noSuccess \right) \right)}{\probability_{\environment}\left( \left( s, \noSuccess \right) \right) + \probability_{\environment}\left( \left( s, \success \right) \right)}\successProbability\left( \policy_{0}, \environment \evolution \left( a, \left( s, \noSuccess \right) \right) \right) \leq\\
      \frac{\probability_{\environment}\left( \left( s, \noSuccess \right) \right)}{\probability_{\environment}\left( \left( s, \noSuccess \right) \right) + \probability_{\environment}\left( \left( s, \success \right) \right)}\successProbability\left( \policy \evolution \left( s, a \right), \environment \evolution \left( a, \left( s, \noSuccess \right) \right) \right),
    \end{multline}
    and apply Lemma \ref{lem.evolvingSingleSuccessAssumesFailure} to both sides getting
    \begin{equation}
      \successProbability\left( \policy_{0}, \environment \evolution \left( a, s \right) \right) \leq \successProbability\left( \policy \evolution \left( s, a \right), \environment \evolution \left( a, s \right) \right)
    \end{equation}
    finishing the proof.
  \end{proof}
\end{theorem}

It might also be worth seeing why we have to be restricted to single success environments in the above theorem, however we will only show this at the end of the last section, when we have the tools to construct a simple example.

Lemma \ref{lem.successRespectsTruncation} and Theorem \ref{thm.pureFiltering}
together suggest the following corollary.

\begin{corollary}
  Let $\policy \colon \constrainedTransducers$ be $\constrainedTransducers$-optimal
  for $\environment \colon \environments$. Then $\policy$ is
  $\constrainedTransducers$-optimal for $\successTruncation\left( \environment
  \right)$, and $\policy \evolution \interleaved{
    \elementString{s}}{\elementString{a} }$ is
  $\constrainedTransducers\evolution
  \interleaved{\elementString{s}}{\elementString{a} }$-optimal for
  $\successTruncation\left( \environment\right)\evolution
  \interleaved{\elementString{a}}{\elementString{s} }$,
  as long as all the evolutions are valid.
  \begin{proof}
      The policy $\policy$ is $\constrainedTransducers$-optimal for the truncation $\successTruncation\left( \environment
  \right)$ by Lemma \ref{lem.successRespectsTruncation}, and the remainder follows directly from Theorem \ref{thm.pureFiltering} since $\successTruncation\left( \environment  \right)$ is a single success environment.
  \end{proof}
\end{corollary}

In particular, this means that an optimal agent can always be seen as optimal
with respect to a single success environment, and performing sensorimotor-only
(i.e. not value-laden) filtering there.

\section{Specifiability}
\label{sec.specifiability}

We will now explore the idea of policies being \emph{specified} by teleo-environments, where we say that a teleo-environment specifies a policy if the policy is uniquely optimal for the given teleo-environment, (possibly within some constrained class of transducers).

A main theme of the section is that within the set of all transducers, the specifiable transducers are exactly the deterministic transducers.
While this result is in keeping with what one might expect from classical
decision theory, it is perhaps counter-intuitive from a cognitive science
perspective, since it suggests that there is no special detectable feature of
those systems whose behaviour are optimal for some sensor-motor task, compared
to those that are not, besides merely being deterministic.

We prove this by showing that, for any given deterministic policy, a teleo-environment can be constructed for which that policy is uniquely optimal.
The way this environment is constructed is rather simple: the agent must behave in exactly the given way, otherwise it will enter the `doom' environment in which success never occurs.
It is of note that while such a simple specifying teleo-environment exists for any deterministic policy, there may be other, much less trivial, environments that also specify the same behaviour.

The equivalence between specifiable transducers and deterministic systems does not hold in every constrained class.
For this reason much of the current section is taken up with providing criteria for constrained classes of transducers within which it does happen.
In Section \ref{sec.absentMinded}
we discuss an example of when it does not happen, which is a version of the famous `absent-minded driver' problem in decision theory.

We begin with the definition of specification:

\begin{definition}[Teleo-specification]
  \label{def.teleoSpecification}
  Let $\constrainedTransducers \subseteq \policies$ be a constrained set of policies. We will say that a teleo-environment $\environment \colon \environments$ \emph{$\constrainedTransducers$-teleo-specifies} a policy $\policy \colon \constrainedTransducers$ iff
  \begin{equation}
    \forall_{\policy' \colon \constrainedTransducers} \successProbability\left( \policy', \environment \right) \geq \successProbability\left( \policy, \environment \right) \implies \policy' = \policy.
  \end{equation}
  In other words $\policy$ is \emph{uniquely} optimal among all the policies in $\constrainedTransducers$. As before we will usually omit the `$\constrainedTransducers$-'.
\end{definition}
We will investigate what conditions are sufficient and necessary for specifiability of policies, that is whether there is an environment that specifies a given policy within its class.

\subsection{Specifying environments}
\label{ssec.specifyingEnv}

Before we investigate properties of specifiable policies, we will present two necessary conditions for an environment to specify a policy among all policies (i.e $\policies$-teleo-specify it). While these conditions are not even close to sufficient, they inform how we approach constructing multiple examples in the following text, so should help the reader to build better intuitions.

In this section we additionally assume that the action space $\actionSpace$ contains at least two elements. Otherwise there is only one policy, and any environment teleo-specifies it. Due to this trivializing nature this could be the assumption throughout the whole paper, but it turns out all other theorems work without it, even if their statement becomes somewhat vacuous.
\begin{proposition}[Uncertain success]
  \label{prop.uncertainSuccess}
  Let $\environment \colon \environments$ teleo-specify the policy $\policy \colon \policies$. For any valid evolution $\left( \star : \left( \elementString{s}:\elementString{t}:\elementString{a} \right) \right) \in \left( \left\{ \star \right\} \times \stateSpace \times \telosSpace \times \actionSpace \right)^{n}$ of the coupled system $\coupling\left( \policy, \environment \right)$, the evolution $\left( \star : \left( \elementString{s}: \noSuccess^n: \elementString{a} \right) \right)$ is also valid.

  \begin{proof}
    Fix an arbitrary order on $\actionSpace$.	Let us define a function
    \begin{equation}
        \zeta \colon \left( \left( \stateSpace \times \actionSpace \right)^{*} \times \policies \right) \to \policies
    \end{equation}
    recursively as
    \begin{multline}
      \zeta\left( \interleaved{ \elementString{s}}{ \elementString{a} }, \policy \right) =\\ \left( \probability_{\policy}, \left( i, o \right) \mapsto \begin{cases}
        \zeta\left( \interleaved{ \elementString{s}_{>1}}{ \elementString{a}_{>1} }, \transition_{\policy}\left( \elementString{s}_{1}, \elementString{a}_{1} \right) \right) & \textrm{ if } \left( i, o \right) = \left( \elementString{s}_{1}, \elementString{a}_{1} \right),\\
        \transition_{\policy}\left( i, o \right) & \textrm{ otherwise}
      \end{cases}\right),
  \end{multline}
  and the recursion base
  \begin{equation}
    \zeta\left( \emptyString, \policy \right) = \left( a_{0}, \left( i, o \right) \mapsto \upsilon \right)
  \end{equation}
  where $a_{0}$ is the least element of $\actionSpace$ such that $\probability_{\policy}\left( a_{0} \right) \neq 1$ (we can choose such an element, because $\actionSpace$ has at least two elements), $\upsilon = \left( U, \left( i, o \right) \mapsto \upsilon \right)$, and $U$ is the uniform distribution on $\actionSpace$. Crucially $\zeta\left( \emptyString, \policy \right) \neq \policy$, because the former assigns probability $1$ to $a_0$, while the latter does not.

  If $\left( \elementString{s}, \elementString{a} \right)$ is a valid evolution of $\policy$, then $\zeta\left( \interleaved{ \elementString{s}}{ \elementString{a} }, \policy \right) \neq \policy$ -- the sequence of transitions which culminates in mapping with the empty string and thus a different transducer, is always within the support of the probability distribution.

  To prove the proposition assume to the contrary, that there exists such an evolution, for which the associated evolution without success is not valid. Since $\left( \star : \left( \elementString{s} : \elementString{t} : \elementString{a} \right) \right)$ is a valid evolution for the whole system, $\interleaved{\elementString{s}}{\elementString{a}}$ is a valid evolution for $\policy$. This means that $\policy' = \zeta\left( \interleaved{ \elementString{s}}{ \elementString{a} }, \policy \right) \neq \policy$, so it suffices to show that it has no lesser success probability to finish the proof.

  Consider equation \eqref{eq.successBellman}. Since $\zeta$ does not change the probabilities for the first $n$ steps and $\successProbability_{n}$ only depends on these steps, $\successProbability_{n}\left( \policy, \environment \right) = \successProbability_{n}\left( \policy', \environment \right)$. In the second sum, the only component influenced by $\zeta$ is the one related to the evolution $\interleaved{\elementString{s}}{\elementString{a}}$, but in this component $\failureProbability_{\coupling\left( \policy, \environment \right)}\left( \elementString{s}, \elementString{a} \right) = 0$, because that is the probability of exactly the environment trajectory $\interleaved{\elementString{a}}{\interleaved{\elementString{s}}{\noSuccess^{n}}}$, which is not a valid evolution, so the probability of it has to be $0$. Thus $\successProbability\left( \policy, \environment \right) = \successProbability\left( \policy', \environment \right)$, ending the proof.
\end{proof}
\end{proposition}
This proposition can be interpreted as the environment ensuring that the teleo-specified agent cannot ever be certain of its success -- intuitively (and in fact also in the proof) in such a situation the agent could behave arbitrarily, contradicting uniqueness. This is clearly a consequence of teleo-specification only caring whether a success occurs at least once.
\begin{proposition}[Everything is possible]
  \label{prop.everythingPossible}
  Let $\environment \colon \environments$ teleo-specify the policy $\policy \colon \policies$. For a valid evolution $\left( \star : \left( \elementString{s}:\elementString{t}:\elementString{a} \right) \right) \in \left( \left\{ \star \right\} \times \stateSpace \times \telosSpace \times \actionSpace \right)^{n}$ of the coupled system $\coupling\left( \policy, \environment \right)$, the associated probability $\probability_{\environment\evolution \interleaved{ \elementString{a}}{ \interleaved{ \elementString{s}}{ \elementString{t} } }}\left( s \right) \neq 0$ for any $s \in \stateSpace$.

  \begin{proof}
    We proceed almost identically as in the previous proof. Assume that there is such an $s \in \stateSpace$ that $\probability_{\environment \evolution \interleaved{ \elementString{a}}{ \interleaved{ \elementString{s}}{ \elementString{t} } }}\left( s \right) = 0$. Also pick an arbitrary $a \in \actionSpace$, such that $\probability_{\policy \evolution \interleaved{ \elementString{s}}{ \elementString{a} }}\left( a \right) \neq 0$. Let $\elementString{s}' = \elementString{s} \cdot s$ and $\elementString{a}' = \elementString{a} \cdot a$. By the choice of $a$, $\policy' = \zeta\left( \interleaved{ \elementString{s}'}{ \elementString{a}' }, \policy \right) \neq \policy$. The rest of the reasoning is exactly as in the previous proof, only the fact that $\failureProbability_{\coupling\left( \policy, \environment \right)}\left( \elementString{s}', \elementString{a}' \right) = 0$ stems from the probability of $s$ being zero, instead of the probability of the failures being zero.
  \end{proof}
\end{proposition}
This proposition can be interpreted as the environment checking the behaviour of the agent in any possible situation -- intuitively the agent could behave arbitrarily conditioning on a sensory input that is impossible, making it not unique.

Note that the environment has to behave this way only on the optimal trajectory -- whenever an agent behaves suboptimally, the environment is not constrained in this way.
\subsection{Specifiable are deterministic}
\label{sec.specifiableAreDeterministic}

We first show that any policy that can be decomposed into a mixture, cannot be specifiable.
\begin{definition}[Nontrivially decomposable]
  We say a constrained policy $\policy \colon \constrainedTransducers \subseteq \policies$ is nontrivially decomposable if there are policies $\policy_{k} \colon \constrainedTransducers \subseteq \policies$ such that $\policy_k\ne \policy$ and numbers $\alpha_{k} \geq 0$ such that $\policy = \mixture{\alpha_}{\policy_}$, and there exists $k_0 \in \N$, such that $\alpha_{k_{0}} \neq 0$.
\end{definition}
\begin{lemma}[Decomposable cannot be specified]
  \label{lem.noSpecifyingDecomposable}
  Consider a constrained policy $\policy \colon \constrainedTransducers \subseteq \policies$ that is nontrivially decomposable. Then $\policy$ is not $\constrainedTransducers$-specifiable.

  \begin{proof}
    We will show there is another policy with no less of a chance of success than $\policy$. By Lemma \ref{lem.successOfMixtures} we can write
    \begin{equation}
      \successProbability\left( \policy, \environment \right) = \mixture{\alpha_}{\politicalSuccessProbability}.
    \end{equation}
    Since the right hand side is a mean of success probabilities, at least one of them is greater or equal to the success probability of $\policy$. It also cannot be equal to $\policy$ by assumption. This contradicts the uniqueness condition for specifiable policies, finishing the proof.
  \end{proof}
\end{lemma}
This result means that if we can show that some transducers within a class can be nontrivially decomposed, then they cannot be specifiable within that class. In particular, deterministic transducers obviously cannot be decomposed, which suggests the following theorem.
\begin{theorem}[Specifiable policies are deterministic]
  \label{thm.specifiableAreDeterministic}
  Consider a class of policies $\constrainedTransducers \subseteq \policies$ in which all nondeterministic transducers are nontrivially decomposable. Then only deterministic policies can be $\constrainedTransducers$-specifiable.

  \begin{proof}
    Follows immediately from Lemma \ref{lem.noSpecifyingDecomposable}.
  \end{proof}
\end{theorem}
However, identifying which transducers can be nontrivially decomposed is not a simple task in general, so it's not obvious which classes of transducers this theorem refers to. In particular it's not \emph{a priori} obvious that all nondeterministic transducers can be nontrivially decomposed in the total class of transducers.
We will now show that this is the case by constructing an explicit decomposition of any nondeterministic transducer (Proposition~\ref{prop.pointwiseDecomposable} below).
\begin{definition}[Pointwise differing transducers]
  \label{def.pointwiseDifferent}
  We say that two transducers $\transducer, \transducer' \colon \genericTransducers$ \emph{differ pointwise} if $\transducer \neq \transducer'$, but there is an $n \in \N$ and a trajectory $\left( \elementString{i}, \elementString{o} \right) \in \inputSpace^{n} \times \outputSpace^{n}$ such that for any $\left( \elementString{j}, \elementString{q} \right) \in \inputSpace^{*} \times \outputSpace^{*} \setminus \left\{ \left( \elementString{i}, \elementString{o} \right) \right\}$ that are valid evolutions for both $\transducer$ and $\transducer'$ we have
  \begin{equation}
    \probability_{\transducer \evolution \interleaved{ \elementString{j}}{ \elementString{q} }} = \probability_{\transducer' \evolution \interleaved{ \elementString{j}}{ \elementString{q} }},
  \end{equation}
  that is they differ only in the probability returned after the single trajectory, in all other cases they are equal.
\end{definition}
Because of the interplay between valid evolutions and associated probabilities, the equality actually works for all valid evolutions of \emph{one} of the transducers that do not have $\left( \elementString{i}, \elementString{o} \right)$ as a prefix.
\newcommand{\twoMixture}[3]{#1 #2 + \left( 1 - #1 \right) #3}
\begin{definition}[Pointwise decomposition]
  \label{def.pointwiseDecomposition}
  We say two pointwise differing transducers $\transducer', \transducer'' \colon \genericTransducers$ and a number $\alpha \in \left( 0, 1 \right)$ constitute a \emph{pointwise decomposition} of $\transducer = \twoMixture{\alpha}{\transducer'}{\transducer''}$.
\end{definition}
An arbitrary nondeterministic transducer is pointwise decomposable within the total class of transducers.
\begin{proposition}[Nondeterministic are pointwise decomposable]
  \label{prop.pointwiseDecomposable}
  We can construct a decomposition as in Definition
  \ref{def.pointwiseDecomposition} for any nondeterministic transducer $\transducer \colon \genericTransducers \setminus \deterministicTransducers$.

  \begin{proof}
    First assume that the probability distribution $\probability_{\transducer}$ associated with the transducer is not a point distribution. We will prove the general case later.

    In this case the probability distribution admits a nontrivial decomposition itself, say of the form
    \begin{equation}
      \probability_{\transducer} = \twoMixture{\alpha}{p}{q}.
    \end{equation}
    Without loss of generality we assume $p$ and $q$ have the same support as $\probability_{\transducer}$.
    Now we can define a pointwise decomposition of the transducer by just using this distribution decomposition and a trivial decomposition on the transition. Thus, the transducers shall be
    \begin{equation}
      \transducer' = \left( p, \transition_{\transducer} \right)
    \end{equation}
    and
    \begin{equation}
      \transducer'' = \left( q, \transition_{\transducer} \right).
    \end{equation}
    The decomposition then takes the shape of
    \begin{equation}
      \transducer = \twoMixture{\alpha}{\transducer'}{\transducer''}.
    \end{equation}
    This equality holds on the mix of probabilities because of their definition and on the transition because the transition functions are equal, so the terms with the normalization in the definition cancel out. It's trivial to check that the transducers differ pointwise for the empty trajectory.

    For the general case note that a nondeterministic transducer has to have a valid evolution $\left( \elementString{i}, \elementString{o} \right) \in \left( \inputSpace \times \outputSpace \right)^{n}$, the result of which $\transducer_0 = \transducer \evolution \interleaved{ \elementString{i}}{ \elementString{o} }$ has a nondeterministic associated distribution. By the previous step we know that $\transducer_0$ pointwise decomposes, so it would be sufficient to prove that if an arbitrary evolution of a transducer pointwise decomposes, then so does the transducer. We will do that coinductively.

    By a similar argument as in the proof of Theorem \ref{thm.filtering} it is sufficient to prove this for the case of $n = 1$ and the rest follows by induction. Let us then consider $\transducer_0 = \transducer \evolution \left( i, o \right)$ such that
    \begin{equation}
      \transducer_0 = \twoMixture{\alpha}{\transducer_0'}{\transducer_0''}
    \end{equation}
    is a pointwise decomposition. We construct a decomposition of the original transducer into
    \begin{equation}
      \transducer' = \left( \probability_{\transducer}, \left( j, q \right) \mapsto \begin{cases}
        \transducer_0' & \textrm{ if } \left( j, q \right) = \left( i, o \right),\\
        \transition_{\transducer}\left( j, q \right) & \textrm{ otherwise}
      \end{cases} \right),
  \end{equation}
  and
  \begin{equation}
    \transducer'' = \left( \probability_{\transducer}, \left( j, q \right) \mapsto \begin{cases}
      \transducer_0'' & \textrm{ if } \left( j, q \right) = \left( i, o \right),\\
      \transition_{\transducer}\left( j, q \right) & \textrm{ otherwise}
    \end{cases} \right).
\end{equation}
The decomposition is then
\begin{equation}
  \transducer = \twoMixture{\alpha}{\transducer'}{\transducer''},
\end{equation}
which adds up correctly because the associated probabilities are identical, as are all transitions other than by $\left( i, o \right)$, and that last one is
\begin{equation}
\begin{split}
  &\frac{\twoMixture{\alpha}{\probability_{\transducer}\left( o \right)\transition_{\transducer'}\left( i, o \right)}{\probability_{\transducer}\left( o \right)\transition_{\transducer''}\left( i, o \right)}}{\twoMixture{\alpha}{\probability_{\transducer}\left( o \right)}{\probability_{\transducer}\left( o \right)}} \\
  &\qquad=   \twoMixture{\alpha}{\transducer_{0}'}{\transducer_{0}''} = \transducer_{0},
\end{split}
\end{equation}
where the first equality is by canceling out the equal probabilities and applying the transition defined above.
This is exactly the transition we would expect.

It is pointwise, because the decomposition of $\transducer_0$ is pointwise by assumption, and any other associated probabilities introduced by this construction are equal between $\transducer'$ and $\transducer''$.
\end{proof}
\end{proposition}
Among the constrained transducer classes we already introduced nondeterministic one-flip transducers are pointwise decomposable (we just need to decompose them arbitrarily at the flip, as the above construction implies), while nondeterministic UFS transducers (including i.i.d. transducers) are not (intuitively, decomposing a maximally complex unifilar $n$ state transducer $\transducer \colon \unifilarFiniteStateTransducers{n} \setminus \unifilarFiniteStateTransducers{n-1}$ would require adding a special memory state for the point of decomposition).

Proposition \ref{prop.pointwiseDecomposable} together with Theorem \ref{thm.specifiableAreDeterministic} implies that only deterministic transducers can be specifiable among unconstrained transducers.

\subsection{Deterministic are specifiable}
\label{sec.deterministicAreSpecifiable}

In this section we show that every deterministic policy is specifiable by explicitly constructing an environment that specifies it. We use a very similar technique as in Example~\ref{ex.noUfsFiltering}, in which we showed that UFS transducers do not satisfy filtering.
\newcommand{\testingEnvironment}[1]{\environment_{\mathrm{test}}\left( #1 \right)}
\begin{definition}[Uniform testing environment]
  \label{def.testingEnvironment}
  Given a deterministic policy $\policy \colon \deterministicTransducers
  \subset \policies$ we define its \emph{uniform testing environment} as
  \begin{multline}
    \testingEnvironment{\policy} =\\
    \left( \frac{1}{4} \nothingDistribution + \frac{3}{4} \successDistribution, \left( a, \left( s, t \right) \right) \mapsto \begin{cases}
      \testingEnvironment{\policy \evolution \left( s, a \right)} & \textrm{ if the evolution is valid,}\\
      \doomEnvironment & \textrm{ otherwise}
    \end{cases} \right),
\end{multline}
where $\nothingDistribution$ and $\successDistribution$ are the uniform nothing distribution and the uniform success distribution (Definition~\ref{def.successDistribution}) and $\doomEnvironment$ is the doom environment (Definition~\ref{def.doomAndDespair}).
For deterministic policies that evolution is valid iff $a$ is the action the policy takes at this step.
\end{definition}
In other words, this environment returns a random state and, as long as the policy behaves like $\policy$, has a $\frac{3}{4}$ chance of returning a success at every step. The exact value of this probability is irrelevant as long as it differs from zero or one; we picked the value $\frac{3}{4}$ to make a future example simpler.
\begin{theorem}[Deterministic policies are specifiable]
  \label{thm.deterministicAreSpecifiable}
  Any deterministic policy $\policy \colon \deterministicTransducers \subset \policies$ is specifiable.

  \begin{proof}
    We will show that $\testingEnvironment{\policy}$ specifies $\policy$. Since this pair has a success chance of $\frac{3}{4}$ at every step we immediately get
    \begin{equation}
      \successProbability\left( \policy, \testingEnvironment{\policy}  \right) = 1.
    \end{equation}
    It suffices to show that any other policy $\policy' \colon \policies$ has a nonzero chance of failure. We will show this inductively with respect to the shortest evolution after which the associated probabilities of the two policies differ.

    If the length is $0$, then $\probability_{\policy} \neq \probability_{\policy'}$. Since $\probability_{\policy}$ is a point distribution, say focused on $a$, this means that there is an action $a'$, such that $\probability_{\policy'}\left( a' \right)$ is not zero. Since the first step results in success with probability $\frac{3}{4}$ and the doom environment never returns success the probability of failure is at least $\frac{1}{4}\probability_{\policy'}\left( a' \right)$, which is not zero.

    For the inductive step we can consider the first step $\left( s, a \right)$ of a $n+1$ length evolution leading to differing probabilities. By definition of the uniform test environment the probability of $(s, \noSuccess)$ in the first step is $\frac{1}{4n}$. The probability of $a$ is $1$, since the distributions have to be equal, because $n+1>0$. Since $\policy' \evolution \left( s, a \right)$ differs from $\policy \evolution \left( s, a \right)$ after an evolution of length $n$ by the inductive assumption it has a nonzero probability of failure $p$. Thus $\policy'$ has a failure probability of at least $\frac{p}{4n}$, which is clearly nonzero.
  \end{proof}
\end{theorem}
In particular this means that for the classes of transducers described in the previous section deterministic transducers exactly coincide with specifiable ones.

\subsection{Optimality of determinism}

Deterministic policies are optimal in another sense -- any environment has a (in general not unique) optimal deterministic policy.

To prove that we will need one more relation on bounded sum sequences and a proof that it respects mixtures and success probability respects it.
\begin{definition}
  \label{def.boundedSumPoset}
  We define the partial order $\geq$ on $\boundedSumSequences{r_{0}}$ coinductively by saying that $\left( r, t \right) \geq \left( r', t' \right)$ iff $\bssSum\left( \left( r, t \right) \right) \geq \bssSum\left( \left( r', t' \right) \right)$ and $t \geq t'$.
\end{definition}
Note that this definition is strictly stronger than just the first sequence having a greater sum -- any suffix of this sequence also has to have a greater sum than the corresponding suffix of the second sequence. The main motivation in defining the relation this way is the ability to use it in coinductive reasoning.
\begin{lemma}
  \label{lem.posetRespectsMixtures}
  Let $\left( r_{k}, t_{k} \right), \left( r'_{k}, t'_{k} \right) \colon
  \boundedSumSequences{s}$ be bounded sum sequences and $\alpha_{k} \in \left[
    0, 1 \right]$ be such that $\sum_{k = 0}^{\infty} \alpha_{k} = 1$. If
  $\forall_{k \in \N} \left( r_{k}, t_{k} \right) \geq \left( r'_{k}, t'_{k}
  \right)$ then
  \begin{equation}
    \mixture{\alpha_}{\kthSuccessSequence} \geq \sum_{k = 0}^{\infty} \alpha_{k} \left( r'_{k}, t'_{k} \right).
  \end{equation}

  \begin{proof}
    To prove that $\mixture{\alpha_}{\kthSuccessSequence} \geq \sum_{k = 0}^{\infty} \alpha_{k} \left( r'_{k}, t'_{k} \right)$ we have to prove that $\Sigma \left(\mixture{\alpha_}{\kthSuccessSequence}\right) \geq \Sigma \left(\sum_{k = 0}^{\infty} \alpha_{k} \left( r'_{k}, t'_{k} \right)\right)$ and that $\sum_{k = 0}^{\infty} \alpha_{k} t_k \ge \sum_{k = 0}^{\infty} \alpha_{k} t'_k$.
    We get the first required inequality by Lemma \ref{lem.sumOfMixtures}, and the second by the coinductive hypothesis.
  \end{proof}
\end{lemma}
\begin{lemma}
  \label{lem.successRespectsPoset}
  The sum of bounded sum sequences respects the order from Definition \ref{def.boundedSumPoset}, that is
  \begin{equation}
    \left( r, t \right) \geq \left( r', t' \right) \implies \bssSum\left( \left( r, t \right) \right) \geq \bssSum\left( \left( r', t' \right) \right).
  \end{equation}

  \begin{proof}
    It suffices to note that the first requirement for inequality of bounded sum sequences is exactly the desired result.
  \end{proof}
\end{lemma}
\begin{theorem}[Deterministic are optimal]
  \label{thm.determinismIsOptimal}
  Let $\environment \colon \environments$ and $\policy \colon \policies$
  be an environment-policy pair. Then there exists a deterministic
  policy $\policy' \colon \deterministicTransducers \subset \policies$ such that
  \begin{equation}
    \successProbability\left( \policy', \environment \right) \geq \successProbability\left( \policy, \environment \right).
  \end{equation}

  \begin{proof}
    \newcommand{\deterministification}{C_{\environment}}
    We will explicitly construct a mapping $\deterministification \colon \policies \to \deterministicTransducers$, such that
    \begin{equation}
      \label{eq.deterministificationIsBetter}
      \successSequence\left( \coupling\left( \deterministification\left( \policy \right), \environment \right) \right) \geq \successSequence\left( \coupling\left( \policy, \environment \right) \right),
    \end{equation}
    so that setting $\policy' = \deterministification\left( \policy \right)$ and applying Lemma \ref{lem.successRespectsPoset} will finish the proof.

    To do that fix a total order on $\actionSpace$ -- we need this only to make one choice unique, so the order does not have to represent anything in particular.
    \newcommand{\optimalAction}{\mathrm{oa}}%
    Define $\optimalAction\left( \policy, \environment \right)$ (for `optimal action') as the element of
    \begin{equation}
      \mathrm{argmax}\left\{ \sum_{s \in \stateSpace} \probability_{\environment}\left( \left( s, \noSuccess \right) \right) \successSequence\left( \coupling\left( \policy \evolution \left( s, a \right), \environment \evolution \left( a, \left( s, \noSuccess \right) \right) \right) \right) \right\}
    \end{equation}
    that is least with respect to this order.
    Now define
    \newcommand{\miniDeterministification}{c_{\environment}}
    \begin{equation}
      \label{eq.determinisation}
      \deterministification\left( \policy \right) = \left( \miniDeterministification\left( \policy \right), \deterministification \circ \transition_{\policy} \right),
    \end{equation}
    where $\miniDeterministification \colon \policies \to \distributions{\actionSpace}$ is defined as
    \begin{equation}
      \miniDeterministification\left( \policy \right) = \begin{cases}
        1 & \textrm{ if } a = \optimalAction\left( \policy, \environment \right),\\
        0 & \textrm{ otherwise.}
      \end{cases}
    \end{equation}
    Note that in \eqref{eq.determinisation}, $\transition_{\policy}$ denotes a restriction of the transition function, as usual.

    It remains to prove that \eqref{eq.deterministificationIsBetter}
    holds. If we write $(r', t') = \successSequence\left(
      \coupling\left( \deterministification\left( \policy \right),
        \environment \right) \right)$ and $(r, t) =
    \successSequence\left( \coupling\left( \policy, \environment \right)
    \right)$ this means we have to prove that $\bssSum\left( (r', t')
    \right) \geq \bssSum\left( (r, t) \right)$ and $t' \geq t$. Let us
    start by considering the second inequality. Expanding the definition
    of the left hand side gives
    \begin{equation}
      t' = \sum_{s \in \stateSpace} \sum_{a \in \actionSpace}
      \probability_{\coupling\left( \deterministification\left( \policy \right), \environment \right)}\left( s, \noSuccess, a \right) \successSequence\left( \coupling\left( \deterministification\left( \policy \right) \evolution \left( s, a \right), \environment \evolution \left( a, \left( s, \noSuccess \right) \right) \right) \right).
    \end{equation}
    Applying the definition of $\deterministification$ to the evolution, while setting $a_{0} = \optimalAction\left( \policy, \environment \right)$ for brevity and noting that the only nonzero probability corresponds to it we have
    \begin{equation}
      t' = \sum_{s \in \stateSpace}
      \probability_{\coupling\left( \deterministification\left( \policy \right), \environment \right)}\left( s, \noSuccess, a_{0} \right) \successSequence\left( \coupling\left( \deterministification\left( \policy \evolution \left( s, a_{0} \right) \right), \environment \evolution \left( a_{0}, \left( s, \noSuccess \right) \right) \right) \right).
    \end{equation}
    We can then use the coinductive hypothesis \eqref{eq.deterministificationIsBetter} to obtain
    \begin{equation}
      t' \geq \sum_{s \in \stateSpace}
      \probability_{\coupling\left( \deterministification\left( \policy \right), \environment \right)}\left( s, \noSuccess, a_{0} \right) \successSequence\left( \coupling\left( \policy \evolution \left( s, a_{0} \right), \environment \evolution \left( a_{0}, \left( s, \noSuccess \right) \right) \right) \right)
    \end{equation}
    and finally by the choice of $a_{0}$ and Lemma \ref{lem.posetRespectsMixtures} this is no less than the weighted mean
    \begin{equation}
      t' \geq \sum_{s \in \stateSpace} \sum_{a \in \actionSpace}
      \probability_{\coupling\left( \policy, \environment \right)}\left( s, \noSuccess, a \right) \successProbability\left( \policy \evolution \left( s, a \right), \environment \evolution \left( a, \left( s, \noSuccess \right) \right) \right) = t,
    \end{equation}
    as required. After applying Lemma \ref{lem.successRespectsPoset} this also gives us
    \begin{equation}
      \label{eq.futureSuccess}
      \bssSum\left( t' \right)\geq \bssSum\left( t \right)
    \end{equation}

    For the first inequality we need notice that
    \begin{equation}
      \sum_{s \in \stateSpace} \sum_{a \in \actionSpace}
      \probability_{\coupling\left( \deterministification\left( \policy \right), \environment \right)}\left( s, \noSuccess, a \right) =
      \sum_{s \in \stateSpace} \sum_{a \in \actionSpace}
      \probability_{\coupling\left( \policy, \environment \right)}\left( s, \noSuccess, a \right),
    \end{equation}
    because the total probability assigned to a state does not depend on the policy. Since $\bssSum\left( \left( r, t \right) \right) = r + \bssSum\left( t \right)$, it remains to use \eqref{eq.futureSuccess} to finish the proof.
  \end{proof}
\end{theorem}

As a simple corollary we get.
\begin{proposition}[Deterministic are uniquely optimal if present]
  \label{prop.deterministicAreBest}
  Let $\constrainedTransducers \subseteq \policies$ be a constrained class of transducers such that $\deterministicTransducers \subseteq \constrainedTransducers$. Then deterministic transducers are exactly $\constrainedTransducers$-teleo-specifiable.

  \begin{proof}
    By Theorem \ref{thm.deterministicAreSpecifiable} they are specifiable, and any nondeterministic transducer cannot be uniquely optimal due to Theorem \ref{thm.determinismIsOptimal}.
  \end{proof}
\end{proposition}
This proposition could have been used to prove that deterministic policies are exactly specifiable policies among unconstrained policies instead of Theorem \ref{thm.specifiableAreDeterministic}. It works in any class of policies that contains all deterministic transducers, even if the nondeterministic transducers are not decomposable in that class. However, Theorem \ref{thm.specifiableAreDeterministic} tells us more about classes which don't necessarily contain all deterministic transducers. For example consider one-flip transducers without deterministic transducers -- Proposition \ref{prop.deterministicAreBest} tells us nothing about specifiability in this class, while Theorem \ref{thm.specifiableAreDeterministic} clearly shows that there are no specifiable policies there at all.

\subsection{The absent-minded driver}
\label{sec.absentMinded}

If we constrain policies to a set in which nondeterministic policies do
not admit nontrivial decompositions, we might encounter policies that are
specifiable, but not deterministic. In this section we translate one
classic example of such a situation into the language of transducers --
the absent-minded driver problem, first proposed in \cite{PICCIONE19973}.
Our version differs somewhat from the original formulation to make it more
suited for our setting and make our analysis simpler, but the reasons
underlying the nondeterminism of the optimal solution remain the same.

In plain language the problem can be stated as follows. There is a motorway with infinitely many identical exits. At every exit the driver can either continue or turn. The driver's goal is to turn at the second exit, but unfortunately they are absent-minded, in the sense that they have no memory at all and hence do not know whether they have already passed any exits or how many. What is the optimal policy for such a driver?

We will need one more generic environment for our example.
\newcommand{\successEnvironment}{\environment_{1}}
\begin{definition}[Success environment]
  \label{def.successEnvironment}
  The \emph{success environment} is one that returns one success and then becomes the doom environment
  \begin{equation}
    \successEnvironment = \left( \successDistribution, \left( a, \left( s, t \right) \right) \mapsto \doomEnvironment \right).
  \end{equation}
  Note that since we only care about success occuring once, the doom environment after transition is not particularly gloomy.
\end{definition}

\begin{example}[The absent-minded driver]
  Let $\actionSpace = \left\{ c, e \right\}$ (where $c$ stands for `continue' and $e$ for `exit'), and $\stateSpace = \left\{ 0 \right\}$. We define the `please exit' environment as
  \begin{equation}
    \environment_{e} = \left( \nothingDistribution, \left( a, \left( s, t \right) \right) \mapsto \begin{cases}
      \successEnvironment & \textrm{ if } a = e,\\
      \doomEnvironment & \textrm{ otherwise}
    \end{cases}\right).
\end{equation}
The absent-minded driver environment is then
\begin{equation}
  \environment = \left( \nothingDistribution, \left( a, \left( s, t \right) \right) \mapsto \begin{cases}
    \environment_{e} & \textrm{ if } a = c,\\
    \doomEnvironment & \textrm{ otherwise}
  \end{cases}\right).
\end{equation}

Making the absent-minded driver actually absent-minded means we restrict our policies to only i.i.d. policies. As mentioned before, they are the same policies as unifilar one state policies $\unifilarFiniteStateTransducers{1}$, and the absent-mindedness translates to only having one memory state.

An i.i.d. policy is uniquely defined by the probabilities it assigns to actions, let us denote them by $p_{c}$ and $p_{e}$ respectively. A simple computation tells us that the success probability of such an i.i.d. policy on $\environment$ is exactly $p_{c}p_{e}$, since the only way of achieving success is to first continue and then exit. The unique values that maximize this probability are $p_{c} = p_{e} = \frac{1}{2}$. Thus the i.i.d. policy assigning these probabilities is $\unifilarFiniteStateTransducers{1}$-teleo-specified by $\environment$, and it is nondeterministic as promised.
\end{example}

It is worth noting an important point about this example.
Since the transducer with  $p_{c} = p_{e} = \frac{1}{2}$ is optimal within the constrained class $\unifilarFiniteStateTransducers{1}$ for the absent-minded driver problem, we could attribute the absent-minded driver problem as a normative-epistemic state.
However, such a normative-epistemic interpretation lacks a property held by the other examples so far, in that its beliefs do not update consistently according to the value-laden filtering scheme described in Section \ref{sec.filtering}.
This lack of temporal consistency with respect to value-laden filtering is a consequence of the fact that $\unifilarFiniteStateTransducers{1}$ does not have the value-laden Bellman property.

\subsection{No sensorimotor-only filtering in generic environments}
\label{sec.noGenericSensorimotor}

Recall that in our setting, the agent's goal is to achieve success at least once, with subsequent success signals counting for nothing.
This leads to the notion of `value-laden filtering', in which the agent always behaves as though it has not already achieved success, since this maximises its overall chances of eventually achieving success.

In light of this we shouldn't expect the evolution of an optimal agent to remain optimal for the environment evolved only by the sensorimotor signals, as that environment becomes a mixture of the evolution after a failure (which the optimal agent cares about) and after a success (which the optimal agent shouldn't care about). If we, then, construct an environment which requires very different behaviour to achieve success after at least one success than otherwise, we should be able to exhibit a counterexample to sensorimotor-only filtering.

Indeed, as promised at the end of the filtering section, we present such an example based on this approach. First let us formally define a special environment behaving as described above.
\newcommand{\trickyTestingEnvironment}[2]{\environment'_{\mathrm{test}}\left( #1, #2 \right)}
\begin{definition}[Tricky testing environment]
  \label{def.trickyTestingEnvironment}
  Given two deterministic policies $\policy, \policy' \in
  \deterministicTransducers \subset \policies$ we define their \emph{tricky testing environment} as
  \begin{multline}
    \trickyTestingEnvironment{\policy}{\policy'} =
    \Bigg( \frac{1}{4} \nothingDistribution + \frac{3}{4} \successDistribution, \\ \left. \left( a, \left( s, t \right) \right) \mapsto \begin{cases}
      \trickyTestingEnvironment{\policy \evolution \left( s, a \right)}{\policy'} & \textrm{ if } t = \noSuccess,\\
      \testingEnvironment{\policy'} & \textrm{ if } t = \success,\\
      \doomEnvironment & \textrm{ if } \policy \evolution \left( s, a \right) \textrm{ is not valid}
    \end{cases} \right),
\end{multline}
where $\nothingDistribution$ and $\successDistribution$ are the uniform nothing distribution and the uniform success distribution (Definition~\ref{def.successDistribution}), $\testingEnvironment{\policy'}$ is the testing environment for $\policy'$ (Definition~\ref{def.testingEnvironment}) and $\doomEnvironment$ is the doom environment (Definition~\ref{def.doomAndDespair}).
\end{definition}
The tricky testing environment $\trickyTestingEnvironment{\policy}{\policy'}$ behaves like the uniform testing environment for $\policy$ until it produces a success, after which it starts behaving like the testing environment for $\policy'$.

It's easy to see that a tricky testing environment is equal to the related uniform testing environment after truncation, that is $\successTruncation\left( \testingEnvironment{\policy} \right) = \successTruncation\left( \trickyTestingEnvironment{\policy}{\policy'} \right)$. In particular, by Lemma \ref{lem.successRespectsTruncation} this means that $\trickyTestingEnvironment{\policy}{\policy'}$ teleo-specifies $\policy$.

\begin{example}[No sensorimotor-only filtering in an arbitrary environment]
  \label{ex.noSensorimotorOnlyFiltering}
  Let $\actionSpace = \left\{ 0, 1 \right\}$, and $\policy, \policy'
  \colon \deterministicTransducers \subset \policies$ be the policies constantly returning $0$ and $1$ respectively. Consider then the environment $\trickyTestingEnvironment{\policy}{\policy'}$. As mentioned above it teleo-specifies $\policy$, in particular this policy is optimal for the environment. However the policy $\policy \evolution \left( s, 0 \right) = \policy$ is \emph{not} optimal for
  \begin{equation}
    \trickyTestingEnvironment{\policy}{\policy'} \evolution \left( 0, s \right) = \frac{1}{4} \trickyTestingEnvironment{\policy \evolution \left( s, 0 \right)}{\policy'} + \frac{3}{4}\testingEnvironment{\policy'}.
  \end{equation}
  To see this it suffices to compute the success probabilities of $\policy$ and $\policy'$ on these environments. Since the policies never return the same value, then by construction of the environments
  \begin{equation}
    \successProbability\left( \policy, \testingEnvironment{\policy'} \right) = \successProbability\left( \policy', \trickyTestingEnvironment{\policy}{\policy'} \right) = \frac{3}{4},
  \end{equation}
  and by the proof of Theorem \ref{thm.deterministicAreSpecifiable} combined with Lemma \ref{lem.successRespectsTruncation}
  \begin{equation}
    \successProbability\left( \policy', \testingEnvironment{\policy'} \right) = \successProbability\left( \policy, \trickyTestingEnvironment{\policy}{\policy'} \right) = 1.
  \end{equation}
  Thus, we can finish the proof using Lemma \ref{lem.successOfMixtures} twice
  \begin{equation}
    \successProbability\left( \policy, \trickyTestingEnvironment{\policy}{\policy'} \evolution \left( 0, s \right) \right) = \frac{13}{16} < \frac{15}{16} = \successProbability\left( \policy', \trickyTestingEnvironment{\policy}{\policy'} \evolution \left( 0, s \right) \right).
  \end{equation}
\end{example}

\section*{ACKNOWLEDGEMENT}
\small
This article was produced with financial and technical support from Principles of Intelligent Behaviour in Biological and Social Systems (PIBBSS), and Simon McGregor's work was supported with a scholarship grant from the Alignment of Complex Systems Research Group (ACS) at Charles University in Prague.
Nathaniel Virgo's work on this publication was made possible through the support of Grant 62229 from the John Templeton Foundation. The opinions expressed in this publication are those of the author(s) and do not necessarily reflect the views of the John Templeton Foundation.

The authors would also like to thank
Martin Biehl for feedback on the manuscript and Robert Obryk for fruitful
discussions about technical details.

\bibliography{notes}

\end{document}